\documentclass{amsart}

\usepackage{amsfonts,amssymb,amsmath,amsthm,amscd,enumerate}
\usepackage{latexsym}
\usepackage{euscript}
\usepackage{setspace}
\usepackage{graphicx}
\usepackage{enumitem}
\usepackage{fullpage}
\newtheorem{teo}{Theorem}[section]
\newtheorem{teo*}{Theorem}
\newtheorem{def.}[teo]{Definition} % o comando [chapter] faz a contagen de teoremas independente por cap\'itulos
\newtheorem{prop}[teo]{Proposition}
\newtheorem{lema}[teo]{Lemma}
\newtheorem{corolario}[teo]{Corollary}
\newtheorem{obs}[teo]{Observation}

\usepackage{indentfirst}
\newcommand{\FF}{F\langle X|G\rangle}
\newcommand{\Z}{\mathbb{Z}}
\newcommand{\N}{\mathbb{N}}
\newcommand{\s}{\mathcal{S}}

\begin{document}

\title[graded codimensions of $E$]{On the $\Z_2$-graded codimensions of the Grassmann algebra over a finite field}
%\author{Eli Aljadeff}\address{Department of Mathematics, Technion-Israel Institute of Technology, Haifa 32000 , Israel}\email{aljadeff@tx.technion.ac.il}

\author{Lucio Centrone}\address{IMECC, Universidade Estadual de Campinas, Rua S\'ergio Buarque de Holanda, 651 � Cidade Universit\'aria ``Zeferino Vaz'' � Distr. Bar\~ao Geraldo � Campinas � S\~ao Paulo � Brasil
CEP 13083-859}\email{centrone@ime.unicamp.br}\author{Lu\'is Felipe
Gon\c{c}alves Fonseca}\address{Instituto de Ci\^encias Exatas e
Tecnol\'ogicas, Universidade Federal de Vi\c{c}osa, Rodovia LMG 818,
km 06, Florestal, Minas Gerais, Brasil. CEP
35690-000}\email{luisfelipe@ufv.br}\begin{abstract} Let $E$ be the infinite dimensional Grassmann
algebra over a finite field $F$ of characteristic not 2. In this
paper we deal with the homogeneous $\Z_2$-gradings of $E$. In
particular, we compute an exact value for the $\Z_2$-graded
homogeneous codimensions of $E$, and a lower and an upper bound for
the $\Z_2$-graded (non-homogeneous) codimensions of $E$ for each of
its $\Z_2$-homogeneous grading.\end{abstract}\maketitle

\section{Introduction}
Polynomial identity algebras have been well investigated since the previous century. We recall that if $X$ is a countable (maybe infinite) set of indeterminates and $F$ is a field, we call \textit{polynomials} the elements of $F\langle X\rangle$, the free algebra freely generated by $X$ over $F$. If $A$ is an $F$-algebra we say $f\in F\langle X\rangle$ is a \textit{polynomial identity} for $A$ if it vanishes under all substitution by elements of $A$. An algebra $A$ is said to be a \textit{PI-algebra} if there exists a non-trivial polynomial identity of $A$. Not to mention the set of polynomial identities of an algebra $A$ is a $T$-ideal called the \textit{$T$-ideal} of $A$ and denoted by $T(A)$. If $A$ is a $G$-graded algebra, where $G$ is a group, we have some similar constructions and we speak about \textit{graded polynomial, graded polynomial identities, graded PI-algebras} and we denote by $T_G(A)$ the set of graded polynomial identities of $A$.

Let $E$ be the Grassmann algebra over $F$. Due to its prominent role in the Kemer's theory about the structure of $T$-ideals (see \cite{kem1}) and its own interest, the (graded) identities of $E$ have been intensively studied. See for example the works by Krakowski and Regev \cite{krr1}, Regev \cite{Regev2}, Giambruno and Koshlukov \cite{gik1}, Anisimov \cite{ani1}, Di Vincenzo and da Silva \cite{did1}, Centrone \cite{Centrone}, Bekh-Ochir \cite{Bekh} and Fonseca \cite{Fonseca}.

It is well known that if $F$ is infinite, every (graded) polynomial identity of a certain PI-algebra $A$ may be recovered by its multihomogeneous (graded) identities. If $F$ is a finite field this is no longer true. See for example the work \cite{Regev2} by Regev and \cite{Bekh} by Bekh-Ochir and Rankin dealing with ordinary polynomial identities of the infinite dimensional Grassmann algebra $E$ and the work \cite{Fonseca} by Fonseca dealing with the $\Z_2$-graded polynomial identities of $E$. Due to this fact, it seems fruitful studying structures related to $T$-ideals explaining the behaviour of the (graded) identities of a given algebra such as \textit{codimensions}.

Let us recall the definition of graded codimension sequence. We
consider the general setting of an $F$-algebra $A$ graded by a
finite group $G=\{g_1,\ldots,g_r\}$ satisfying a graded polynomial
identity; then we consider $X=\bigcup_{g\in G}X^g$ be a finite
(maybe infinite) set of indeterminates labeled by the elements of
the group, such that $|X^g|=l_g$ and $X^g\cap X^h=\emptyset$ if
$g\neq h$. We consider $\FF$ which is the graded generalization of
the algebra of non-commutative polynomials. Let us
consider a homogeneous subspace $V$ of $\FF$. We define the
\textit{$(n_1,\ldots,n_r)$-th $V$ codimension} of $A$ as
$c_{(n_1,\ldots,n_r)}(A,V):=\dim_FV_{(n_1,\ldots,n_r)}/(V_{(n_1,\ldots,n_r)}\cap
T_G(A))$, where $V_{(n_1,\ldots,n_r)}$ is the vector space of the elements of $V$ of total degree
$n_i$ in the variables of $X^{g_i}$.

In this paper we deal with homogeneous $\Z_2$-graded codimensions of
$E$. We recall we have three up to isomorphism kinds of homogenous
$\Z_2$-gradings which are used to be denoted by $E_{\infty}$,
$E_{k^*}$ and $E_k$ (see \cite{ani1} and \cite{did1}) or $E$ if we consider any of
them without distinguish them. We firstly consider $V$ being the space of homogeneous polynomials of $F\langle
X|\Z_2\rangle$ and we obtain that $c_{(n_1,n_2)}(E,V)$ is limited
for each $n_1,n_2\in\N$, then we start to study the structure of
such subspaces. In particular, we found a basis for
$V_{(n_1,n_2)}/(V_{(n_1,n_2)}\cap T_{\Z_2}(E))$, when $V$ is the space of multihomogeneous polynomials, and an exact value
for $c_{(n_1,n_2)}(E,V)$. In the last part we consider a finite
dimensional generating subspace $W$ of $F\langle X|\Z_2\rangle$ and
we give upper and lower bounds for $c_{(n_1,n_2)}(E,W(n))$, where
$W(n)=W^{\otimes n}$. We recall that this is a $\Z_2$-graded
generalization of the work \cite{Regev2} by Regev in which he found
an exact value for the multihomogeneous and multilinear codimensions
of $E$ in the non-graded case. We want to point out our techniques
are different than the Regev's ones.

\section{Basic tools}

All algebras we refer to are associative and unitary unless explicitely written. For the sake of simplicity all groups are to be considered finite.

%Moreover all the fields are to be considered infinite and of characteristic different from 2.

\begin{def.}
Let $(G,\cdot)=\{g_1,\ldots,g_r\}$ be any group and let $F$ be a field. If $A$ is an $F$-algebra, we say that $A$ is a
$G$-graded algebra if there are subspaces $A^g$ for each $g\in G$ such that \[A=\bigoplus_{g\in G}A^g \ \textrm{and} \
A^gA^h\subseteq A^{gh}\] for every $g,h\in G$. If $0\neq a\in A^g$ we say that $a$ is \textit{homogeneous of $G$-degree $g$} or simply that $a$ has degree $g$,
and we write $\deg(a)=g$.\end{def.}

We consider now some special class of gradings.

\begin{def.}
Let $A$ be an $F$-algebra and $G$ be a group. A $G$-grading on $A$ is said to be \textit{homogeneous} if there exists a basis $\mathcal{B}=\{b_1,\ldots,b_n,\ldots\}$ of $A$ as a vector space and a map $\phi:\mathcal{B}\rightarrow G$ such that $\deg(b_i)=\phi(b_i)$ for each $i$.
\end{def.}

One defines \textit{$G$-graded subspaces} of $A$, \textit{$G$-graded $A$-modules}, \textit{$G$-graded homomorphisms} and so on, in a standard way (see for
example \cite{giz3} for more details).

Let $\{X^{g}\mid g \in G\}$ be a family of disjoint countable sets. Set $X=\bigcup_{g\in G}X^{g}$ and denote by $\FF$ the free-associative algebra freely generated by the set $X$. An indeterminate $x\in X$ is said to be of \textit{homogeneous $G$-degree $g$}, written $\deg(x)=g$,
if $x\in X^{g}$. We always write $x^{g}$ if $x\in X^{g}$. The homogeneous $G$-degree of a monomial $m=x_{i_1}x_{i_2}\cdots x_{i_k}$ is defined to be
$\deg(m)=\deg(x_{i_1})\cdot\deg(x_{i_2})\cdot\cdots\cdot\deg(x_{i_k})$. For every $g \in G$, we denote by $\FF^g$ the subspace of $\FF$ spanned by all the
monomials having homogeneous $G$-degree $g$. Notice that $\FF^g\FF^{g'}\subseteq \FF^{gg'}$ for all $g,g' \in G$. Thus \[\FF=\bigoplus_{g\in G}\FF^g\] is a $G$-graded algebra. The elements of the $G$-graded algebra $\FF$ are referred to as \textit{$G$-graded polynomials} or, simply, \textit{graded polynomials}. We refer to the multihomogeneous degree of a polynomial as the usual multidegree of $\FF$.

\begin{def.} If $A$ is a $G$-graded algebra, we denote by $T_{G}(A)$ the intersection of the kernels of all
$G$-graded homomorphisms $\FF\rightarrow A$. Then $T_{G}(A)$ is a graded two-sided ideal of $\FF$ and its elements are called \textit{$G$-graded polynomial identities} of the algebra $A$.\end{def.}

Notice that $T_{G}(A)$ is stable under the action of any $G$-graded endomorphism of the algebra $\FF$.
Any $G$-graded ideal of $\FF$ which verifies such property is said to be a \textit{$T_{G}$-ideal}. Clearly,
any $T_{G}$-ideal $I$ is the ideal of the $G$-graded polynomial identities of the graded algebra
$\FF/I$. Note also that for a $G$-graded algebra $A$, the quotient algebra $\FF/T_G(A)$ is the
relatively-free algebra for the variety of $G$-graded algebras generated by $A$. If $G$ is finite of order $r$ and $X=\bigcup_{g\in G}X^g$ is finite, we shall denote the relatively free graded algebra of $A$ by $U_{l_1,\ldots,l_r}(A)$, where $l_i=|X_i|$.

\

If $S\subseteq \FF$, we shall denote by $\langle S\rangle^{T_G}$ the \textit{$T_G$-ideal generated by the set $S$}, i.e., the smallest $T_G$-ideal containing $S$. Moreover, if $I=\langle \mathcal{S}\rangle^{T_G}$ we say
$\mathcal{S}$ is a basis for $I$ or that the elements of $I$ \textit{follow from} or \textit{are consequences of} the elements of $\mathcal{S}$.

We introduce now the codimension sequence.

\begin{def.}
Let $G=\{g_1,\ldots,g_r\}$, $X$ be a finite set such that $|X^g|=l_g$ and consider $\FF$. Let $V\subseteq\FF$ be a linear $G$-homogeneous subspace and $(n_1,\ldots,n_r)\in\N^r$. Let $A$ be a $G$-graded PI-algebra, then we define the \textit{$(n_1,\ldots,n_r)$-th $V$ codimension} of $A$ as $c_{(n_1,\ldots,n_r)}(A,V):=\dim_FV_{(n_1,\ldots,n_r)}/(V_{(n_1,\ldots,n_r)}\cap T_G(A))$, where $V_{(n_1,\ldots,n_r)}$ is the vector space over $F$ of the elements of $V$ of total degree $n_i$ in the variables of $X^{g_i}$.
\end{def.}
\

In this paper we are going to deal with $\mathbb{Z}_2$-gradings over the infinite dimensional Grassmann algebra. Sometimes we use the word \textit{superalgebra} instead of $\mathbb{Z}_2$-graded algebra. From now on we shall refer to the $y_i$'s as variables of $\Z_2$-degree 0 and to the $z_j$'s as variables of $\Z_2$-degree 1; $X=Y\cup Z$, where $|Y|=l$, $|Z|=m$, and we shall write $F\langle X\rangle$ instead of $F\langle X|\Z_2\rangle$. Moreover we will use the notation $T_2(A)$ instead of $T_{\Z_2}(A)$ for any superalgebra $A$.

\begin{def.}
Let $V=\{v_1,v_2,\ldots\}$ be an infinite countable set, then we denote by $E=E(V)$ the Grassmann algebra generated by $V$, i.e. $F\langle V\rangle/I$, where for each $i$, $e_i=v_i+I$ and $I$ is the ideal generated by $\{v_iv_j+v_jv_i|i,j\in\N\}$. It is well known that $B_E=\left\{e_{i_1}e_{i_2}\cdots e_{i_n}\mid n\in \N, {i_1}<{i_2}<\cdots<{i_n}\right\}$ is a basis of $E$ as a vector space over $F$. Moreover we say $e_{i_1}\cdots e_{i_l}$ is a basis element of $E$ of \textit{length} $l$ with \textit{support} $\{e_{i_1},\ldots,e_{i_l}\}$.
\end{def.}

Let us consider the map $\varphi:V\rightarrow \Z_2$ such that $v_i\mapsto 1$. The map $\varphi$ gives out a $\Z_2$-grading over $E$ called \textit{canonical grading}. In this case, let $E^0$ be the homogeneous component of $\Z_2$-degree 0 and let $E^1$ be the component of degree 1. It is easy to see that $E^0$ is the center of $E$ and $ab + ba = 0$ for all $a, b\in E^1$. This means that $E$ satisfies the following graded polynomial identities:
$[y_1, y_2]$, $[y_1, z_1]$, $z_1z_2 + z_2z_1$.
Now, let us consider the homogeneous $\Z_2$-gradings over $E$. We recall the latter are induced by the maps  $\deg_{k*}$, $\deg_\infty,$ and $\deg_k,$ defined respectively by

$$\deg_{k*}(e_i)=\left\{\begin{array}{ll}
\text{\rm 1 for $i = 1,\ldots, k$}\\
\text{\rm 0 otherwise},\end{array}\right.$$

$$\deg_{\infty}(e_i)=\left\{\begin{array}{ll}
\text{\rm 1 for $i$ odd}\\
\text{\rm 0 otherwise},\end{array}\right.$$

$$\deg_k(e_i)=\left\{\begin{array}{ll}
\text{\rm 0 for $i = 1,\ldots, k$}\\
\text{\rm 1 otherwise}.\end{array}\right.$$

From now on we shall denote by $E_{k^*}$, $E_\infty$, $E_k$ the Grassmann algebra endowed with the $\Z_2$-grading induced by the maps $\deg_{k*}$, $\deg_\infty,$ and $\deg_k$. We denote by $E$ any of the superalgebras $E_{k^*}$, $E_\infty$, $E_k$ without distinguish them.

\
%\begin{remark}\label{possiblechoose}
%We note that we may always choose elements of $E_d^0$ of the form $e_{i_1}e_{i_2}+\cdots+e_{i_{2a-1}}e_{i_{2a}}$ for each $a\in\N$.
%\end{remark}

We introduce a special type of polynomials that turned out to be crucial when describing the $\Z_2$-graded identities of $E$ (see \cite{did1}). Let us consider $[x_{1},x_{2}] := x_{1}x_{2} - x_{2}x_{1}$ the Lie comutator of $x_{1}$ with $x_{2}$. We shall define inductively
$[x_{1},\ldots,x_{n-1},x_{n}] := [[x_{1},\ldots,x_{n-1}],x_{n}], \ \
n = 3,4,5,\ldots$.

\

Let $f=z_{i_1}^{r_{i_1}}\cdots z_{i_s}^{r_{i_s}}[z_{j_1},z_{j_2}]\cdots [z_{j_{t-1}},z_{j_t}]$ and consider the set \[\s=\{\text{different homogeneous variables appearing in $f$}\}\subseteq\{z_1,\ldots,z_m\}.\] If $h=|\s|,$ then $\s=\{z_{i_1},\ldots,z_{i_h}\}.$  Notice also that $f$ is linear in the commutators.

We consider now \[T = \{j_1,\ldots, j_t\}\subseteq\s\] and let us denote the previous polynomial by \[f_T (z_{i_1},\cdots, z_{i_h}).\]

\begin{def.} For $m\geq2$ let $$g_m(z_{i_1},\ldots, z_{i_h}) =\sum_{\begin{array}{cc}
T\\
\text{\rm $|T|$ even}\end{array}}(-2)^{-\frac{|T|}{2}}f_T(z_{i_1},\ldots, z_{i_h}),$$
moreover put
$g_1(z_1) = z_1.$\end{def.}

The following is well known.

\begin{lema}\label{id1} Let $E$ be the Grassmann algebra over any field, then $[x_{1},x_{2},x_{3}] \in T_{2}(E)$.
\end{lema}

Keeping in mind Lemma 1.4.2 of \cite{Drensky}, we have
$[x_{1},x_{2}][x_{3},x_{4}] - [x_{1},x_{3}][x_{2},x_{4}] \in \langle
[x_{1},x_{2},x_{3}]\rangle_{T_{2}}$. Notice also that
$[x_{2},x_{1},\ldots,x_{1}]$, where $x_{1}$ appears $p$-times in the
brakets of the last comutator, follows from $[x_{1},x_{2},x_{3}]$.
If $char F = p$, we have $- [x_{2},x_{1},\ldots,x_{1}] =
[x_{1}^{p},x_{2}]$. Hence $[x_{1}^{p},x_{2}]$ follows from
$[x_{1},x_{2},x_{3}]$ when $char F = p > 2$ (see Lemma 3.3 of
\cite{Bekh}). Thus, we have the following lemma.

\begin{lema}\label{id1.5}
The polynomials $[x_{1}^{p},x_{2}], [x_{1},x_{2}][x_{3},x_{4}] -
[x_{1},x_{3}][x_{2},x_{4}]$ belong to $\langle
[x_{1},x_{2},x_{3}]\rangle_{T_{2}}$.
\end{lema}

From now on we shall denote by $F$ a finite field such that $char F = p >
2$ and $|F| = q$, by $E$ the infinite dimensional Grassmann algebra with unit $1_E$ generated by $\{e_{1},e_{2},\ldots,e_{n},\ldots\}$. The non-unitary infinite dimensional Grassmann algebra will be denoted by $E^*$. Moreover we shall denote by $B$ the linear basis $\{e_{i_{1}}\cdots e_{i_{n}}| i_{1} < \ldots < i_{n}, n \geq 0 \}$ of $E$.

The next is an easy consequence of Lemma 1.2-b and Corollary 1.5-a of \cite{Regev2}.

\begin{lema}\label{id2}
Let us consider $\lambda.1_{E} + a \in E$, where  $a \in E^{*}, \lambda
\in F$. Then $a^{p} = \lambda^{p}.1_{E}$. Moreover, $y_{1}^{pq} -
y_{1}^{p} \in T_{2}(E)$.
\end{lema}

In light of Lemma \ref{id2}, we have $z_{1}^{p} \in T_{2}(E)$.

\begin{def.}

Let us consider $a = e_{i_{1}}\cdots e_{i_{n}} \in B\cap E^{*}$. We shall define the \textit{support} of $a$ as the set $supp(a) = \{e_{i_{1}},\ldots,e_{i_{n}}\}$.
We define the \textit{length-support} of $a$ as $wt(a): = |supp(a)|$. We set
$supp(1_{E}) = \emptyset$ and $wt(1_{E}) = 0$.
\end{def.}

\begin{def.}
For every $g = \sum_{i=1}^{n} \lambda_{i}a_{i} \in E - \{0\}$,
we define the \textit{support} of $g$ as $supp(g) :=
\cup_{i=1}^{n}supp(a_{i})$, and the \textit{length-support} of $g$ as
$wt(g):= max\{wt(a_{i})| i = 1,\cdots,n\}$. We shall define the \textit{dominating part} of $g$
as $dom(g) := \sum_{wt(a_{i}) = wt(g)}
\lambda_{i}a_{i}$.
\end{def.}

Let $B = \{y_{1},\cdots,y_{n},\cdots, z_{1},\cdots,z_{n},
\cdots,\cdots,[x_{1},x_{2}],[x_{1},x_{3}],\cdots,\newline
[x_{2},x_{3}],[x_{2},x_{4}],\cdots,
[x_{j_{1}},x_{j_{2}},x_{j_{3}}],\cdots,\cdots,[x_{j_{1}},\cdots,x_{j_{n}}],\cdots\}$
be an ordered linear basis for the subspace generated by $X$ and
commutators $[x_{i_{1}},\cdots,x_{i_{n}}], n = 2,3,\cdots$.

By Poincar\'e-Birkhoff-Witt's theorem we have the following polynomials

\begin{center}\label{equa}
$x_{i_{1}}^{a_{1}}\ldots
x_{i_{n_{1}}}^{a_{n_{2}}}[x_{j_{1}},\ldots,x_{j_{l}}]^{b_{1}}\ldots[x_{r_{1}},\ldots,x_{r_{t}}]^{b_{n_{2}}}$,
\end{center}
where $a_{1},\ldots,a_{n_{1}},b_{1},\ldots,b_{n_{2}}$ are
non-negative integers, form a linear basis of $F\langle X \rangle$.
$x_{i_{1}},\cdots,x_{i_{n_{1}}},[x_{j_{1}},\cdots,x_{j_{l}}],
\cdots,\newline [x_{r_{1}},\cdots,x_{r_{t}}] \in B$, and $x_{i_{1}}
< \cdots < x_{i_{n_{1}}} < [x_{j_{1}},\cdots,x_{j_{l}}] < \cdots <
[x_{r_{1}},\cdots,x_{r_{t}}]$.

From now on every polynomial of $F\langle X \rangle/T_2(E)$ will be
written as a linear combination of elements of the latter basis
which will be denoted by $Pr(X)$. Of course, due to Lemmas \ref{id1}
and \ref{id1.5} every element of $Pr(X)$ may be written modulo the
graded identities as a linear combination of polynomials of type
\[(\prod_{r = 1}^{n} y_{j_{r}}^{a_{j_{r}}})(\prod_{r = 1}^{m}
z_{i_{r}}^{b_{i_{r}}})[x_{t_{1}},x_{t_{2}}]\cdots
[x_{t_{2l-1}},x_{t_{2l}}] \in Pr(X).\]

We consider now the following definition.

%Let $V\subseteq\FF$, where $V=H_{(l,m)}(F,n_1,n_2)$ the set of homogeneous polynomials of degree $n_1$ in the $Y$'s and degree $n_2$ in the $Z$'s.

\begin{def.}
Let $a = (\prod_{r = 1}^{n} y_{j_{r}}^{a_{j_{r}}})(\prod_{r =
1}^{m} z_{i_{r}}^{b_{i_{r}}})[x_{t_{1}},x_{t_{2}}]\cdots
[x_{t_{2l-1}},x_{t_{2l}}] \in Pr(X)$. We shall denote:
\begin{description}
\item $beg(a): = (\prod_{r = 1}^{n} y_{j_{r}}^{a_{j_{r}}})(\prod_{r =
1}^{m} z_{i_{r}}^{b_{i_{r}}})$ and $\psi(a) := x_{t_{1}}\cdots
x_{t_{2l}}$;
\item $\Pi(Y)(a) := (\prod_{r = 1}^{n} y_{j_{r}}^{a_{j_{r}}})$ and
$\Pi(Z)(a) := (\prod_{r = 1}^{m} z_{i_{r}}^{b_{i_{r}}})$;
\item $pr(z)(a) = z_{i_{1}}$ (if $\Pi(Z)(a) \neq 1$);
\item $Deg_{x_{i}} a$: the number of times in which the variable $x_{i}$
appears in $beg(a)\psi(a)$;
\item $deg_{Y} a:= \sum_{y \in Y} Deg_{y}(a)$, $deg_{Z} a:= \sum_{z \in Z} Deg_{z}(a)$
and $deg a := deg_{Z} a + deg_{Y} a$;
\item $\mathcal{V}(a) := \{x \in X | Deg_{x}(a) > 0\}$;
\item $Yyn(a) := \{x \in \mathcal{V}(a) \cap Y| Deg_{x}(beg (a)) > 0, Deg_{x}(\psi(a))
= 0\}$;
\item $Yyy(a) := \{x \in \mathcal{V}(a) \cap Y| Deg_{x}(beg (a)) > 0, Deg_{x}(\psi(a))
> 0\}$;
\item $Yny(a) := \{x \in \mathcal{V}(a) \cap Y| Deg_{x}(beg (a)) = 0, Deg_{x}(\psi(a))
> 0\}$;
\item $Zyn(a) := \{x \in \mathcal{V}(a) \cap Z| Deg_{x}(beg (a)) > 0, Deg_{x}(\psi(a))
= 0\}$;
\item $Zyy(a) := \{x \in \mathcal{V}(a) \cap Z| Deg_{x}(beg (a)) > 0, Deg_{x}(\psi(a))
> 0\}$;
\item $Zny(a) := \{x \in \mathcal{V}(a) \cap Z| Deg_{x}(beg (a)) = 0, Deg_{x}(\psi(a))
> 0\}$.
\end{description}
\end{def.}

\begin{def.}
A linear combination of elements of $Pr(X)\cap F\langle
y_{1},\ldots,y_{n}\rangle$ $f = \sum_{j=1}^{l}\lambda_{j}m_{j}$, where $\psi(m_{1}) = \cdots = \psi(m_{l}) = 1$, is called \textit{$p$-polynomial} if $Deg_{y_{i}} m_{j} \equiv 0$ mod $p$ and $Deg_{y_{i}} m_{j} < qp$ for every $i \in \{1,\ldots,n\}$ e $j \in
\{1,\ldots,l\}$. The vector space of $p$-polynomials in the variables $y_{1},\ldots,y_{n}$ will be denoted by $ppol(y_{1},\ldots,y_{n})$.
\end{def.}
\begin{obs}
It is easy to note that $dim (ppol(y_{1},\ldots,y_{n})) = q^{n}$.
\end{obs}

The next two results may be found in \cite{Fonseca} (see
\cite{Fonseca}, Proposition 4.5 and Corollary 4.6).

\begin{prop}\label{id5}
If $f \in T_{2}(E)$ is a $p$-polynomial, then $f$ is 0.
\end{prop}

\begin{corolario}\label{cor1}
If $f(y_{1},\ldots,y_{m})$ is a non-zero $p$-polynomial, then there
exist $\alpha_{1},\ldots,\alpha_{m} \in F$ such that
$f(\alpha_{1}.1_{E},\ldots,\alpha_{m}.1_{E}) \neq 0$.
\end{corolario}

\section{The set $SS$ and its total order}

We shall construct a subset of $Pr(X)$ that will be useful in the descriptions of the relatively free $\mathbb{Z}_2$-graded algebra of $E$ that we are going to study through the paper.

\begin{def.}
We say $a \in Pr(X)$ belongs to $SS$ if $Deg_{x}beg(a) \leq
p - 1$ for every $x \in X$ and $\psi(a) = 1$ or $a$ is multilinear.
\end{def.}

In \cite{Fonseca} Proposition 6.1 the author describes the relatively free algebra \[\frac{F\langle X \rangle}{\langle
[x_{1},x_{2},x_{3}], z_{1}^{p}, y_{1}^{pq} - y_{1}^{p}
\rangle^{T_{\Z_2}}}.\] In particular, we have the next result (see \cite{Fonseca} proposition 6.1).

\begin{prop}\label{id3}
Let $f = \sum_{i=1}^{n}\lambda_{i}v_{i}$ be a linear combination of elements of $Pr(X)$. Then $f$ may be written modulo $\langle [x_{1},x_{2},x_{3}],
z_{1}^{p}, y_{1}^{pq} - y_{1}^{p} \rangle_{T_{2}}$, as:
\begin{center}
$\sum_{i=1}^{m}f_{i}u_{i}$,
\end{center}
where $f_{1},\ldots,f_{m}$ are $p$-polynomials and
$u_{1},\ldots,u_{m}$ are distinct elements of $SS$.
\end{prop}

We are going to order a subset of $SS$ with the right lexicographic order while the total order on $SS$ was presented in \cite{Fonseca}.

\begin{def.}
Let $u,v \in SS$ such that $\psi(u) = \psi(v) = 1$. We say $u
<_{lex-rig} v$ when $Deg_{x_{i}} u < Deg_{x_{i}} v$ for some
$x_{i} \in X$ and $Deg_{x} u = Deg_{x} v$ for every $x
> x_{i}$ (with respect to the ordered basis given by $X$ and the commutators).
\end{def.}

\begin{def.}
Given $u,v \in SS$, we say $u < v$ when:
\begin{description}
\item $deg u < deg v$ or
\item $deg u = deg v$, but $beg(u) <_{lex-rig} beg(v)$ or
\item $deg u = deg v, beg(u) = beg(v)$, but $\psi(u) <_{lex-rig}
\psi(v)$.
\end{description}
\end{def.}

\begin{def.}
Let $f = \sum_{i = 1}^{n}\lambda_{i}u_{i}$ be a linear combination of distinct elements of $SS$. We shall call \textit{leading term} of $f$, denoted by $LT(f)$, the element $u_{i} \in \{u_{1},\ldots,u_{n}\}$ such that
$u_{j} \leq u_{i}$ for every $j \in \{1,\ldots,n\}$.
\end{def.}

\begin{def.}
Let $f = \sum_{i = 1}^{n}\lambda_{i}u_{i}$ a linear combination of distinct elements of $SS$. We shall call $u_{i}$ \textit{bad} if the following conditions hold:
\begin{description}
\item $Deg_{x}(u_{i}) = Deg_{x}(LT(f))$ for every $x \in X$;
\item If $deg_{Z}(beg(LT(f))) > 0$ and $z \in Z - \{pr(z)(LT(f))\}$, then $Deg_{z} beg(LT(f)) = Deg_{z}
beg(u_{i})$;
\item If $deg_{Z}(beg(LT(f))) > 0$ and $z = pr(z)(LT(f))$, then $Deg_{z}(beg(u_{i})) + 1 =
Deg_{z}(beg(LT(f)))$;
\item For every $x \in Y$, we have $Deg_{x} beg(LT(f)) \leq Deg_{x}
beg(u_{i})$.
\end{description}
If $f$ has a bad term, we shall denote by $LBT(f)$ its greater bad term.
\end{def.}
\begin{obs}\label{observacao}
Notice that if $f$ has a bad term $u_{i}$, then $u_{i} < LT(f)$.
Moreover, there exists a variable $x \in Y$ such that
$Deg_{x}(beg(LT(f))) < Deg_{x}(beg(u_{i}))$.
\end{obs}

\begin{def.}
Let $u \in SS$. We say $u$ is of \textit{Type-$0$} (or $u \in SS0$) if the following conditions hold:
\begin{description}
\item $\psi(u) = 1$;
\item $\Pi(Z)(u)$ is $1$ or a multilinear polynomial.
\end{description}

\end{def.}

In 8.1 de \cite{Fonseca}, Theorem 8.1 the author shows up a basis for the graded identities of $E_{can}$. In particular, we have the following.

\begin{teo}\label{ecan}
The $\mathbb{Z}_{2}$-graded identities of $E_{can}$ follow from the identities:
\begin{center}
$y_{1}y_{2} - y_{2}y_{1} , z_{1}z_{2} + z_{2}z_{1}, y_{1}z_{2} -
z_{2}y_{1}$ e $y_{1}^{pq} - y_{1}^{p}$.
\end{center}
\end{teo}

Notice that $z_{1}^{2}$ is a consequence of $z_{1}z_{2} +
z_{2}z_{1}$. Hence $z_{1}^{p} \in T_{2}(E_{can})$. Note also that
$2z_{1}z_{2} = [z_{1},z_{2}]$ modulo $ T_{2}(E_{can})$ and
$[x_{1},x_{2},x_{3}] \in \langle y_{1}y_{2} - y_{2}y_{1} ,
z_{1}z_{2} + z_{2}z_{1}, y_{1}z_{2} - z_{2}y_{1} \rangle_{T_{2}}$.
In light of the previous comments and Theorem \ref{ecan}, it follows
easily the next result.

\begin{prop}\label{oo}
Let $f = \sum_{i=1}^{n}\lambda_{i}v_{i}$ be a linear combination of elements of $Pr(X)$. Then $f$ may be written modulo $\langle y_{1}y_{2} - y_{2}y_{1} ,
z_{1}z_{2} + z_{2}z_{1}, y_{1}z_{2} - z_{2}y_{1}, y_{1}^{pq} -
y_{1}^{p} \rangle_{T_{2}}$, as:
\begin{center}
$\sum_{i=1}^{m}f_{i}u_{i}$,
\end{center}
where $f_{1},\ldots,f_{m}$ are $p$-polynomials and
$u_{1},\ldots,u_{m}$ are distinct elements of $SS0$.
\end{prop}

In \cite{Fonseca} Theorema 8.2, the author shows a basis for the
$\mathbb{Z}_{2}$-graded identities of $E_{\infty}$. In particular we
have the following.

\begin{teo}\label{teorema 8.2}
The $\mathbb{Z}_{2}$-graded identities of $E_{\infty}$ follow from:
\begin{center}
$[x_{1},x_{2},x_{3}], z_{1}^{p}$ e $y_{1}^{pq} - y_{1}^{p}$.
\end{center}
\end{teo}

In \cite{Fonseca} Theorem 9.3, the author shows a basis for the $\mathbb{Z}_{2}$-graded identities of $E_{k^{*}}$, where $k \geq 1$.

In particular, we have the following result.

\begin{teo}\label{teorema 9.3}
The $\mathbb{Z}_{2}$-graded identities of $E_{k^{*}}$ follow from:
\begin{center}
$[x_{1},x_{2},x_{3}], z_{1}^{p}, z_{1}\cdots z_{k+1}$ e $y_{1}^{pq}
- y_{1}^{p}$.
\end{center}
\end{teo}

When $k < p$, we have $z_{1}^{p}$ follows from
$z_{1}.\cdots.z_{k+1}$.

In \cite{Bekh} Theorem 3.1, Ochir and Rankin showed a basis for the
ordinary identities of $E$. By the latter result, it is easy to show
that the $\mathbb{Z}_{2}$-graded identities of $E_{0^{*}}$ follow
from the polynomials $[y_{1},y_{2},y_{3}],y_{1}^{pq} - y^{p},
z_{1}$. Notice that $z_{1}^{p}$ follows from $z_{1}$.

\begin{def.}
An element $a \in SS$ is said to be of \textit{Type-1} (or $u \in
SS1$) if $deg_{Z}(a) \leq k$.
\end{def.}

Due to the identity $z_{1}\cdots z_{k+1} \in T_{2}(G_{k^{*}})$, we
have another version of Proposition \ref{id3} for $E_{k^{*}}$. The
next result is also true in the case $k = 0$ (see \cite{Fonseca}
Proposition 9.2).

\begin{prop}\label{proposicao 9.2}
Let $f = \sum_{i=1}^{n}\lambda_{i}v_{i}$ be a linear combination of elements of $Pr(X)$. Then $f$ may be written modulo $\langle [x_{1},x_{2},x_{3}],
z_{1}^{p}, y_{1}^{pq} - y_{1}^{p}, z_{1}\cdots
z_{k+1}\rangle_{T_{2}}$, as:
\begin{center}
$\sum_{i=1}^{m}f_{i}u_{i}$,
\end{center}
where $f_{1},\ldots,f_{m}$ are $p$-polynomials and $u_{1},\ldots,u_{m}
\in SS1$ are distinct.
\end{prop}

In \cite{Fonseca} Theorem 10.17, the author describes a basis for
the $\mathbb{Z}_{2}$-graded identities of $E_{k}$ when $k \geq 1$.
In particular we have the next result.

\begin{teo}\label{teorema 10.17}
The $\mathbb{Z}_{2}$-graded identities of $E_{k}$ follow from the following polynomial identities:
\begin{itemize}
\item $[y_{1},y_{2}]\cdots [y_{k},y_{k+1}]$ (if $k$ is odd) \ \ (1);
\item $[y_{1},y_{2}]\cdots [y_{k-1},y_{k}][y_{k+1},x]$ (if $k$
is even and $x \in X - \{y_{1},\ldots,y_{k+1}\})$ \ \ (2);
\item $[x_{1},x_{2},x_{3}]$ \ \ (3);
\item $g_{k-l+2}(z_{1},\cdots,z_{k-l+2})[y_{1},y_{2}]\cdots
[y_{l-1},y_{l}]$ (if $l \leq k$ and $l$ is even) \ \ (4);
\item $g_{k-l+2}(z_{1},\ldots,z_{k-l+2})[z_{k-l+3},y_{1}][y_{2},y_{3}]\cdots[y_{l-1},y_{l}]$
(if $l \leq k$ and $l$ is odd) \ \ (5);
\item
$[g_{k-l+2}(z_{1},\ldots,z_{k-l+2}),y_{1}]\cdots[y_{l-1},y_{l}]$ (if
$l\leq k$ and $l$ is odd) \ \ (6);
\item $z_{1}^{p}$ \ \ (7);
\item $y_{1}^{pq} - y_{1}^{p}$ \ \ (8).
\end{itemize}
\end{teo}

\begin{def.}
An element $u_{i} \in SS$ is said to be of \textit{Type-$2$} (or $u_{i} \in SS2$)
if the following condition holds:
\begin{description}
\item $deg_{Y}(\psi(u_{i})) \leq k$ and $deg_{Z}(beg(u_{i})) + deg_{Y}(\psi(u_{i})) \leq
k+1$.
\end{description}
\end{def.}

\begin{def.}
An element $u_{i} \in SS$ is said to be of \textit{Type-$3$} (or $u_{i} \in SS3$)
if the following conditions hold:
\begin{description}
\item $u_{i} \in SS2$;
\item If $deg_{Z} (beg(u_{i})) + deg_{Y} (\psi(u_{i})) = k + 1$, then
$Deg_{pr(z)(u_{i})} \psi(u_{i}) = 0$.
\end{description}
\end{def.}

Now we have the next result (see \cite{Fonseca} Proposition 10.16).

\begin{prop}\label{proposicao 10.16}
Let $f = \sum_{i=1}^{n}\lambda_{i}v_{i}$ be a linear combination of elements of $Pr(X)$. Then $f$ may be written modulo $T_{2}(E_{k})$, as:
\begin{center}
$\sum_{i=1}^{m}f_{i}u_{i}$,
\end{center}
where $f_{1},\ldots,f_{m}$ are $p$-polynomials and
$u_{1},\ldots,u_{m}$ are distinct elements of $SS3$.
\end{prop}

\section{Bounds for $\Z_2$-graded codimensions of $E$}

In this section we are going to give an upper and a lower bound for the $\Z_2$-graded codimension of $E$ endowed with a homogeneous $\Z_2$-grading. We shall start with some general facts whereas in the next sections we study each case separately. 

In the sequel $F$ will denote a finite field of characteristic $p>2$ and order $q$ unless explicitely written. We recall that if $A$ is a $\Z_2$-graded PI-algebra we denote by $U_{l,m}(A)$ its $\Z_2$-graded relatively free algebra in $l$ variables of degree 0 and $m$ variables of degree 1.

\begin{prop}\label{cod2}
Let $F$ be a finite field such that $|F|=q$, then \[\dim_FU_{l,m}(E)\leq\frac{(l+m)^{pql+pm+1}-1}{l+m-1},\] where $E$ is assumed $\Z_2$-graded by an homogeneous $\Z_2$-grading.
\end{prop}
\proof
Let $W$ be a linear subspace of $F\langle X\rangle$ such that $\dim_FW=l+m$ and $F\langle X\rangle=\bigoplus_{d\geq0}(W^{\otimes d})$. We consider $R_n=\bigoplus_{d=0}^n(W^{\otimes d})$, then $\dim_FR_n=\sum_{d=0}^n(l+m)^d=\frac{(l+m)^{n+1}-1}{l+m-1}$. Let $f\in F\langle X\rangle$ and suppose $f$ multihomogeneous, then $f$ may be written as \[y_1^{a_1}\cdots y_l^{a_l}z_1^{b_1}\cdots z_m^{b_m}g(y_1,\ldots,y_l,z_1,\ldots,z_m),\] where $g(y_1,\ldots,y_l,z_1,\ldots,z_m)$ is multilinear. By Lemma \ref{id2} there exist $a'_i$'s and $b'_j$'s, where $a'_i<pq$ and $b'_j<p$, such that $f$ is $y_1^{a'_1}\cdots y_l^{a'_l}z_1^{b'_1}\cdots z_m^{b'_m}g(y_1,\ldots,y_l,z_1,\ldots,z_m)$ modulo its graded identities. This means $f\in R_{pql+pm}$ modulo the identities and the assertion follows.
\endproof

The previous result gives us that the $\Z_2$-graded Gelfand-Kirillov dimension of $E$ in a fixed number of graded variables is 0. See \cite{cen4} for more details about the graded Gelfand-Kirillov dimension of graded algebras and \cite{Centrone} for a comparison with the case of $E$ over an infinite field. Now we start to focus on codimensions. In what follows $V$ will denote the space of homogeneous polynomials of $F\langle X\rangle$.

\begin{prop}\label{cod1}
Let $F$ be a field (maybe infinite) of characteristic $p\neq0$, then if $n_2\geq m(p+1)$ we have $c_{(n_1,n_2)}(E,V)=0$.
\end{prop}
\proof
It is sufficient to observe that the result is true if it is true on multihomogeneous polynomials. Hence let $f=f(y_1,\ldots,y_l,z_1,\ldots,z_m)$ be multihomogeneous such that $\deg_Yf=n_1$ and $\deg_Zf=n_2$, then we may assume \[f=y_1^{a_1}\cdots y_l^{a_l}z_1^{b_1}\cdots z_m^{b_m}g(y_1,\ldots,y_l,z_1,\ldots,z_m),\] where $g(y_1,\ldots,y_l,z_1,\ldots,z_m)$ is multilinear. Then $b_1+\cdots+b_m+m=\deg_Zf=n_2$ which implies $b_1+\ldots+b_m=n_2-m\geq mp$, hence some of the $b_i$'s is greater than $p$ and the assertion follows because of Lemma \ref{id2}.
\endproof

By Proposition \ref{cod2} we have the following.

\begin{corolario}
Fon any $n_1,n_2\in\N$ we have \[c_{(n_1,n_2)}(E,V)\leq\frac{(l+m)^{pql+pm+1}-1}{l+m-1}.\]
\end{corolario}

Let $W$ be any generating subspace of $F\langle X\rangle$ and do consider $W(n)=W^{\otimes n}$. Let $n_1,n_2\in\N$ such that $n_1+n_2=n$. Because of Lemma \ref{id2} we have the $T_2$-ideal of $E$ is not homogeneous we cannot recover the graded codimensions $c_{(n_1,n_2)}(E,W(n))$ from the homogeneous ones, then we have to study them one by one.

\

For further use we recall the following combinatorial tool (see the book of Stanley \cite{sta1}).

\begin{prop}
The number of commutative monomials of degree $n$ in $k$ variables such that each variable has degree strictly less than $j$ is given by \[\kappa(n,j,k)=\sum_{r+sj=n}(-1)^s{k+r-1\choose r}{k\choose s}.\]
\end{prop}

We may also count the number of $p$-polynomials of a certain degree.

\begin{lema}\label{countingppolynomials}
Let $s$ be an integer. If $p$ divides $s$, the number of $p$-polynomials in $l$ variables of degree $s$ is \[p(s):=q^{\kappa(s/p,q,l)},\]otherwise it is 0 and we write $p(s)=1$.
\end{lema}

Moreover, we have the following combinatorial lemmas counting the number of polynomials of $SSi$ having degree $n_1$ with respect to even variables e degree $n_2$ with respect to the odd ones.

\begin{lema}\label{combss}
Let $n_1,n_2\in\N$, then the number of elements of $SS$ having 0-degree $n_1$ and 1-degree $n_2$ is given by:\[c_{(n_1,n_2)}(SS)=\sum_{s=0}^{\left\lfloor \frac{n_1+n_2}{2}\right\rfloor}{n_1+n_2\choose 2s}+\left\{\kappa(n_1,p,l)\kappa(n_2,p,m)\right\},\]where the last summand appears for non-multilinear elements of $SS$.
\end{lema}
%
%\begin{lema}\label{combss0}
%Let $n_1,n_2\in\N$, then the number of elements of $SS0$ having 0-degree $n_1$ and 1-degree $n_2$ is given by:\[c_{(n_1,n_2)}(SS0)=\kappa(n_1,p,l)\kappa(n_2,2,m).\]
%\end{lema}
%
%\begin{lema}\label{combss1}
%Let $n_1,n_2,k\in\N$, where $n_2\leq k$, then the number of elements of $SS1$ having 0-degree $n_1$ and 1-degree $n_2$ is given by:\[c_{(n_1,n_2)}(SS1)=\sum_{\beta=0}^k\sum_{s=0}^{\left\lfloor \frac{l+m}{2}\right\rfloor}{m\choose\beta}{l\choose 2s-\beta}\kappa(n_1-2s+\beta,p,l)\kappa(n_2-\beta,p,m).\]
%\end{lema}
%
%\begin{lema}\label{combss2}
%Let $n_1,n_2,k\in\N$, then the number of elements of $SS2$ having 0-degree $n_1$ and 1-degree $n_2$ is given by:\[\sum_{\alpha=0}^k\sum_{s\geq\frac{2\alpha+n_2-k-1}{2}}^{\left\lfloor \frac{l+m}{2}\right\rfloor}{l\choose\alpha}{m\choose 2s-\alpha}\kappa(n_1-\alpha,p,l)\kappa(n_2-2s+\alpha,p,m).\]
%\end{lema}
%
%\begin{lema}\label{combss3}
%Let $n_1,n_2,k\in\N$, then the number of elements of $SS3$ having 0-degree $n_1$ and 1-degree $n_2$ is given by:\[c_{(n_1,n_2)}(SS3)=\sum_{\alpha=0}^k\sum_{s\geq\frac{2\alpha+n_2-k}{2}}^{\left\lfloor \frac{l+m}{2}\right\rfloor}{m\choose\beta}{l\choose 2s-\beta}\kappa(n_1-2s+\beta,p,l)\kappa(n_2-\beta,p,m)+\]\[\sum_{\alpha=0}^k{l\choose\alpha}{m\choose 2\xi-\alpha}\kappa(n_1-\alpha,p,l)\kappa(n_2-2\xi,p,m),\]where $\xi=\frac{2\alpha+n_2-k-1}{2}$.
%\end{lema}

\begin{lema}\label{combss0}
Let $n_1,n_2\in\N$, then the number of elements of $SS0$ having 0-degree $n_1$ and 1-degree $n_2$ is given by:\[c_{(n_1,n_2)}(SS0)=\kappa(n_1,p,l)\kappa(n_2,2,m).\]
\end{lema}

%\begin{lema}\label{combss1}
%Let $n_1,n_2,k\in\N$, where $n_2\leq k$, then the number of elements of $SS1$ having 0-degree $n_1$ and 1-degree $n_2$ is given by:\[c_{(n_1,n_2)}(SS1)=\sum_{\beta=0}^k\sum_{s=0}^{\left\lfloor \frac{l+m}{2}\right\rfloor}{m\choose\beta}{l\choose 2s-\beta}\kappa(n_1-2s+\beta,p,l)\kappa(n_2-\beta,p,m).\]
%\end{lema}

\begin{lema}\label{combss1}
Let $n_1,n_2,k\in\N$, where $n_2\leq k$, then the number of elements of $SS1$ having 0-degree $n_1$ and 1-degree $n_2$ is given by:\[c_{(n_1,n_2)}(SS1)=\sum_{2s=0}^{n_1+n_2}\sum_{
\begin{array}{c}
    \beta\leq k\\
    \beta\leq 2s
\end{array}
}{m\choose\beta}{l\choose 2s-\beta}{m-\beta\choose n_2-\beta}{l-2s+\beta\choose n_1-2s+\beta}+\left\{\kappa(n_2,p,m)\kappa(n_1,p,l)\right\},\]where the last summand appears for non-multilinear elements of $SS1$.
\end{lema}

%\begin{lema}\label{combss2}
%Let $n_1,n_2,k\in\N$, then the number of elements of $SS2$ having 0-degree $n_1$ and 1-degree $n_2$ is given by:\[\sum_{\alpha=0}^k\sum_{s\geq\frac{2\alpha+n_2-k-1}{2}}^{\left\lfloor \frac{l+m}{2}\right\rfloor}{l\choose\alpha}{m\choose 2s-\alpha}\kappa(n_1-\alpha,p,l)\kappa(n_2-2s+\alpha,p,m).\]
%\end{lema}

\begin{lema}\label{combss2}
Let $n_1,n_2,k\in\N$, then the number of elements of $SS2$ having 0-degree $n_1$ and 1-degree $n_2$ is given by:\[c_{(n_1,n_2)}(SS2)=\sum_{2s=0}^{n_1+n_2}\sum_{
\begin{array}{c}
    \beta\leq k\\
    \beta\leq 2s\\
    \beta\leq \frac{k+1+2s-n_2}{2}
\end{array}
}{l\choose\beta}{m\choose 2s-\beta}{l-\beta\choose n_1-\beta}{m-2s+\beta\choose n_2-2s+\beta}+\left\{\kappa(n_2,p,m)\kappa(n_1,p,l)\right\},\]
where the last summand appears for non-multilinear elements of $SS2$ when $n_2\leq k+1.$\end{lema}

%\begin{lema}\label{combss3}
%Let $n_1,n_2,k\in\N$, then the number of elements of $SS3$ having 0-degree $n_1$ and 1-degree $n_2$ is given by:\[c_{(n_1,n_2)}(SS3)=\sum_{\alpha=0}^k\sum_{s\geq\frac{2\alpha+n_2-k}{2}}^{\left\lfloor \frac{l+m}{2}\right\rfloor}{m\choose\beta}{l\choose 2s-\beta}\kappa(n_1-2s+\beta,p,l)\kappa(n_2-\beta,p,m)+\]\[\sum_{\alpha=0}^k{l\choose\alpha}{m\choose 2\xi-\alpha}\kappa(n_1-\alpha,p,l)\kappa(n_2-2\xi,p,m),\]where $\xi=\frac{2\alpha+n_2-k-1}{2}$.
%\end{lema}

\begin{lema}\label{combss3}
Let $n_1,n_2,k\in\N$, then the number of elements of $SS3$ having 0-degree $n_1$ and 1-degree $n_2$ is given by:\[c_{(n_1,n_2)}(SS3)=\sum_{2s=0}^{n_1+n_2}\sum_{
\begin{array}{c}
    \beta\leq k\\
    \beta\leq 2s\\
    \beta<\frac{k+1+2s-n_2}{2}
\end{array}
}{l\choose\beta}{m\choose 2s-\beta}{l-\beta\choose n_1-\beta}{m-2s+\beta\choose n_2-2s+\beta}\]\[+\sum_{2s=0}^{n_1+n_2}\sum_{
\begin{array}{c}
    \beta\leq k\\
    \beta\leq 2s\\
    \beta=\frac{k+1+2s-n_2}{2}
\end{array}
}{l\choose\beta}{m-1\choose 2s-\beta}{l-\beta\choose n_1-\beta}{m-2s+\beta\choose n_2-2s+\beta}+\left\{\kappa(n_2,p,m)\kappa(n_1,p,l)\right\},\]
where the last summand appears for non-multilinear elements of $SS3$ when $n_2\leq k+1.$\end{lema}

In light of the above lemmas, the next ones count the number of polynomials $f_iu_i$, where $f_i$ is a $p$-polynomial and $u_i$'s are element of $SSj$, having degree $n_1$ with respect to even variables e degree $n_2$ with respect to the odd ones.

\begin{lema}\label{combssstar}
Let $n_1,n_2\in\N$, then the number of elements of the type $fu_i$, where $f$ is a monomial which is a $p$-polynomial and $u_i$'s are element of $SSj$, having degree $n_1$ with respect to even variables e degree $n_2$ with respect to the odd ones is given by:
\[c_{(n_1,n_2)}^*(SSj):=\sum_{
\begin{array}{c}
    s_i\leq q-1\\
    p\sum s_i\leq n_1
\end{array}
}c_{(n_1-p\sum s_i,n_2)}(SSj).\]
\end{lema}

\begin{lema}\label{combssstar2}
Let $n_1,n_2\in\N$, then the number of elements of the type $fu_i$, where $f$ is a $p$-polynomial and $u_i$'s are element of $SSj$, having degree $n_1$ with respect to even variables e degree $n_2$ with respect to the odd ones is given by:
\[c_{(n_1,n_2)}^{\circ}(SSj):=\sum_{
\begin{array}{c}
    s\leq n_1
\end{array}
}p(s)c_{(n_1-s,n_2)}(SSj).\]
\end{lema}

\section{$\Z_2$-graded homogeneous codimensions}
We start with the $\Z_2$-graded homogeneous codimensions. In this section $V$ will denote the set of multihomogeneous polynomials. In this case we are able to find an exact value for $c_{(n_1,n_2)}(E,V)$ for each homogeneous $\Z_2$-grading of $E$.

\

Let $Multi(a_{1},\ldots,a_{l},b_{1},\ldots,b_{m})$ be the set of multihomogeneous polynomials of ${F \langle X
\rangle}$ in the variables $y_{1},\ldots,y_{l},z_{1},\ldots,z_{m}$ and multidegree $(a_{1},\ldots,a_{l},b_{1},\ldots,b_{m})$ and let us denote by \[MultiFree(a_{1},\ldots,a_{l},b_{1},\ldots,b_{m}):=Multi(a_{1},\ldots,a_{l},b_{1},\ldots,b_{m})/(Multi(a_{1},\ldots,a_{l},b_{1},\ldots,b_{m})\cap T_2(E)).\]

Let us note that
\begin{center}
$1 \leq a_{1},\ldots,a_{l} \leq pq $ and $1\leq b_{1},\ldots,b_{m}
\leq p$.
\end{center}

Then from now on if $(a_{1},\ldots,a_{l},b_{1},\ldots,b_{m})$ is the multidegree of a multihomogeneous polynomial, we tacitely assume $a_{1},\ldots,a_{l},b_{1},\ldots,b_{m}$
are bounded as above.

\subsection{$E_{can}$}

\

\begin{def.}
Let us denote by $pPol-SS0(a_{1},\ldots,a_{l},b_{1},\ldots,b_{m})$
the set of polynomials of type $fu \in
Multifree(a_{1},\ldots,a_{l},b_{1},\ldots,b_{m})$, where \text{\rm
$f$ is a monomial $p$-polynomial with coefficient 1, $u_i\in SS0$.}
\end{def.}

It is not difficult to see that if $fu,f'u' \in
Multifree(a_{1},\ldots,a_{l},b_{1},\ldots,b_{m})$ are such that $u =
u'$, then $f = f'$.

As a consequence of Proposition \ref{oo}, we have
$pPol-SS0(a_{1},\ldots,a_{l},b_{1},\ldots,b_{m})$ is a generating set of the vector space
$Multifree(a_{1},\ldots,a_{l},b_{1},\ldots,b_{m})$. In what follows we shall prove that the set of polynomials $pPol-SS0(a_{1},\ldots,a_{l},b_{1},\ldots,b_{m})$ is a linearly independent set.

\begin{prop}
The set $pPol-SS0(a_{1},\ldots,a_{l},b_{1},\ldots,b_{m})$ is linearly independent.
\end{prop}
\begin{proof}
Let us suppose $f\in
pPol-SS0(a_{1},\ldots,a_{l},b_{1},\ldots,b_{m})$ such that $f =
\sum_{i = 1}^{n}\gamma_{i}f_{i}u_{i} = 0$ where $\gamma_{i} \neq 0$
for some $i \in \{1,\ldots,n\}$. Let us suppose, without loss of
generality, $u_{i}$ is the biggest term of $\{u_{1},\ldots,u_{n}\}$.

Due to Corollary \ref{cor1} there exist
$\lambda_{1},\ldots,\lambda_{l} \in F$ such that
$f_{i}(\lambda_{1}1_{E},\ldots,\lambda_{l}1_{E}) \neq 0$.

Let us consider now the following graded homomorphism:

\begin{center}
$\phi : F \langle y_{1},\ldots,y_{l},z_{1},\ldots,z_{m} \rangle
\rightarrow E$ \\
$z_{1} \mapsto e_{1}$ \\
$\ldots$ \\
$z_{m} \mapsto e_{m}$ \\
$y_{1} \mapsto \lambda_{1}1_{E} + e_{m+1}e_{m+2} + \ldots + e_{m +
2Deg_{y_{1}}(u_{i}) - 1}e_{m + 2Deg_{y_{1}}(u_{i})}$\\
$\ldots$ \\
$y_{l} \mapsto \lambda_{l}1_{E} + e_{m + 2(Deg_{y_{1}}(u_{i}) +
\ldots + Deg_{y_{l - 1}}(u_{i})) + 1}e_{m + 2(Deg_{y_{1}}(u_{i}) +
\ldots + Deg_{y_{l - 1}}(u_{i})) + 2} + \ldots + e_{m +
2(Deg_{y_{1}}(u_{i}) + \ldots + Deg_{y_{l}}(u_{i})) - 1}e_{m +
2(Deg_{y_{1}}(u_{i}) + \ldots + Deg_{y_{l}}(u_{i}))}$
\end{center}

Notice that $dom(\phi(\gamma_{i}f_{i}u_{i})) = \beta\gamma_{i}
e_{1}.\ldots.e_{m + 2(Deg_{y_{1}}(u_{i}) + \ldots +
Deg_{y_{l}}(u_{i}))}$ for some $\beta \in F - \{0\}$. Moreover, if
$j \neq i$, we have $wt(\phi(f_{j}u_{j})) < m + 2(Deg_{y_{1}}(u_{i})
+ \ldots + Deg_{y_{l}}(u_{i}))$. Hence
\begin{center}
$dom(\phi(f)) = \lambda\alpha_{i}e_{1}.\ldots.e_{m +
2(Deg_{y_{1}}(u_{i}) + \ldots + Deg_{y_{l}}(u_{i}))} \neq 0$,
\end{center}
which is a contradiction and we are done.
\end{proof}

\begin{corolario}\label{impcan}
The set $pPol-SS0(a_{1},\ldots,a_{l},b_{1},\ldots,b_{m})$ is a basis for $Multifree(a_{1},\ldots,a_{l},b_{1},\ldots,b_{m})$.
\end{corolario}

By Corollary \ref{impcan} and Lemma \ref{combssstar} we have the next.

\begin{teo*}
Let $n_1,n_2\in\N$, then \[c_{(n_1,n_2)}(E_{can},V)=c_{(n_1,n_2)}^*(SS0).\]
\end{teo*}

\subsection{$E_{\infty}$}

\

\begin{def.}
Let us denote by $pPol-SS(a_{1},\ldots,a_{l},b_{1},\ldots,b_{m})$
the set of polynomials of type $fu \in
Multifree(a_{1},\ldots,a_{l},b_{1},\ldots,b_{m})$, where \text{\rm
$f$ is a monomial $p$-polynomial with coefficient 1, $u_i\in
SS$.}\end{def.}

It is not difficult to see that if $fu,f'u' \in
Multifree(a_{1},\ldots,a_{l},b_{1},\ldots,b_{m})$ are such that $u =
u'$, then $f = f'$.

In the sequel we shall always assume
\[Yyn(LT(f)),Yyy(LT(f)),Zyn(LT(f)),Zyy(LT(f)),Zny(LT(f)),Yny(LT(f))\]
being non-empty. For the sake of convenience, we shall assume each
of the first four sets having at least two elements. We shall denote
them by:
\begin{description}
\item $Yyn(LT(f)) = \{y_{1},\ldots,y_{l_{1}}\}$,
\item $Yyy(LT(f)) = \{y_{n_{1} + 1},\ldots,y_{l_{1} +
l_{2}}\}$,\item $Yny(LT(f)) = \{y_{l_{1} + l_{2} +
1},\ldots,y_{l_{1} + l_{2} + l_{3}}\}$,\item $Zyn(LT(f)) =
\{z_{1},\ldots,z_{m_{1}}\}$,\item $Zyy(LT(f)) =
\{z_{m_{1}+1},\ldots,z_{m_{1}+m_{2}}\},$ \item $Zny(LT(f)) =
\{z_{m_{1}+m_{2}+1},\ldots,z_{m_{1}+m_{2}+m_{3}}\}$.
\end{description}

As a consequence of Proposition \ref{id3}, we have
$pPol-SS(a_{1},\ldots,a_{l},b_{1},\ldots,b_{m})$ is a generating set
for the vector space
$Multifree(a_{1},\ldots,a_{l},b_{1},\ldots,b_{m})$. We shall prove
$pPol-SS(a_{1},\ldots,a_{l},b_{1},\ldots,b_{m})$ is a linearly
independent set.

\begin{prop}
The set $pPol-SS(a_{1},\ldots,a_{l},b_{1},\ldots,b_{m})$ is linearly independent.
\end{prop}
\begin{proof}
Set $l=l_1+l_2+l_3$ and $m=m_1+m_2+m_3$.
Let us suppose $f\in pPol-SS(a_{1},\ldots,a_{l},b_{1},\ldots,b_{m})$, such that $f = \sum_{i=1}^{n}\gamma_{i}f_{i}u_{i} =
0$, where $\gamma_{i} \neq 0$ for some $i \in \{1,\ldots,n\}$.
We shall denote by $LT(f)$ the biggest term of
$\{u_{1},\ldots,u_{n}\}$. We shall suppose, without loss of generality, $u_{i} = LT(f)$.

Due to Corollary \ref{cor1} there exist
$\lambda_{1},\ldots,\lambda_{l} \in F$ such that
$f_{i}(\lambda_{1}1_{E},\ldots,\lambda_{l}1_{E}) \neq 0$.

We shall denote by $\alpha_{1} = 4(Deg_{y_{1}}(LT(f)) + \ldots +
Deg_{y_{l_{1}}}(LT(f)))$, $\alpha_{2} = \alpha_{1} + 2(Deg_{y_{l_{1}
+ 1}}(LT(f)) + \ldots + Deg_{y_{l_{1} + l_{2}}}(LT(f))) + 2(l_{2} -
l_{1})$, $\alpha_{3} = \alpha_{2} + 2.l_{3}$, $\alpha_{4} =
\alpha_{3} + 2.(Deg_{z_{1}}(LT(f)) + \ldots +
Deg_{z_{m_{1}}}(LT(f)))$, $\alpha_{5} = 2.(Deg_{z_{1}}(LT(f)) +
\ldots + Deg_{z_{m_{1}}}(LT(f))) - 1$, $\alpha_{6} = \alpha_{5} + 2(
Deg_{z_{m_{1} + 1}}(LT(f)) + \ldots + Deg_{z_{m_{1} +
m_{2}}}(LT(f)))$, $\alpha_{7} = \alpha_{6} + 2.m_{3}$ e $\alpha_{8}
= \alpha_{3} + \alpha_{5} + 2.(Deg_{z_{m_{1} + 1}}(LT(f)) + \ldots +
Deg_{z_{m_{1} + m_{2}}}(LT(f))) + 2(2 - (m_{2} + 1)) - 1$.

Let us consider the following graded homomorphism $\phi: F\langle
y_{1},\ldots,y_{l_{1} + l_{2} + l_{3}}, z_{1}, \ldots, z_{m_{1} +
m_{2} + m_{3}} \rangle \rightarrow E$:

\begin{center}
$y_{1} \mapsto \lambda_{1}1_{E} + e_{2}e_{4} + \ldots +
e_{4.Deg_{y_{1}}(LT(f)) - 2}e_{4.Deg_{y_{1}}(LT(f))}$ \\

$\ldots$ \\

$y_{l_{1}} \mapsto \lambda_{l_{1}}1_{E} + e_{4(Deg_{y_{1}}(LT(f)) +
\ldots + Deg_{y_{l_{1} - 1}}(LT(f))) + 2}e_{4(Deg_{y_{1}}(LT(f)) +
\ldots + Deg_{y_{l_{1} - 1}}(LT(f))) + 4} + \ldots +
e_{4(Deg_{y_{1}}(LT(f)) + \ldots + Deg_{y_{l_{1}}}(LT(f))) -
2}e_{4(Deg_{y_{1}}(LT(f)) + \ldots +
Deg_{y_{l_{1}}}(LT(f)))}$ \\

$y_{l_{1} + 1} \mapsto \lambda_{l_{1} + 1}1_{E} +  e_{\alpha_{1} +
2} + e_{\alpha_{1} + 4}e_{\alpha_{1} + 6} + \ldots + e_{\alpha_{1} +
4. Deg_{y_{l_{1} +
1}}(LT(f))}e_{\alpha_{1} +  4.Deg_{y_{l_{1} + 1}}(LT(f)) + 2}$ \\

$\ldots$ \\

$y_{l_{1} + l_{2}} \mapsto \lambda_{l_{1} + l_{2}}1_{E} +
e_{\alpha_{1} + 4(Deg_{y_{l_{1} + 1}}(LT(f)) + \ldots +
Deg_{y_{l_{1} + l_{2} - 1}}(LT(f))) + 2(l_{2} - l_{1})} +
e_{\alpha_{1} + 4(Deg_{y_{l_{1} + 1}}(LT(f)) + \ldots +
Deg_{y_{l_{1} + l_{2} - 1}}(LT(f))) + 2(l_{2} - l_{1}) +
2}e_{\alpha_{1} + 4(Deg_{y_{l_{1} + 1}}(LT(f)) + \ldots +
Deg_{y_{l_{1} + l_{2} - 1}}(LT(f))) + 2(l_{2} - l_{1}) + 4} + \ldots
+ e_{\alpha_{1} + 4(Deg_{y_{l_{1} + 1}}(LT(f)) + \ldots +
Deg_{y_{l_{1} + l_{2}}}(LT(f))) + 2(l_{2} - l_{1} - 1)}e_{\alpha_{1}
+ 4(Deg_{y_{l_{1} + 1}}(LT(f)) + \ldots
+ Deg_{y_{l_{2} + l_{1}}}(LT(f))) + 2(l_{2} - l_{1})}$ \\

$y_{l_{1} + l_{2} + 1} \mapsto \lambda_{l_{1} + l_{2} + 1}1_{E} + e_{\alpha_{2} + 2}$ \\

$\ldots$ \\

$y_{l_{1} + l_{2} + l_{3}} \mapsto \lambda_{l_{1} + l_{2} + l_{3}}1_{E} + e_{\alpha_{2} + 2.l_{3}}$ \\

$z_{1} \mapsto e_{1}e_{\alpha_{3} + 2} + \ldots +
e_{2.(Deg_{z_{1}}(LT(f))) - 1}e_{\alpha_{3} + 2(Deg_{z_{1}}(LT(f)))}$\\

$\ldots$ \\

$z_{m_{1}} \mapsto e_{2.(Deg_{z_{1}}(LT(f)) + \ldots + Deg_{z_{m_{1}
- 1}}(LT(f))) + 1}e_{\alpha_{3} + 2.(Deg_{z_{1}}(LT(f)) + \ldots +
Deg_{z_{m_{1} - 1}}(LT(f))) + 2} + \ldots + e_{2.(Deg_{z_{1}}(LT(f))
+ \ldots + Deg_{z_{m_{1}}}(LT(f))) - 1}e_{\alpha_{3} +
2.(Deg_{z_{1}}(LT(f)) + \ldots +
Deg_{z_{m_{1}}}(LT(f)))}$ \\

$z_{m_{1} + 1} \mapsto e_{\alpha_{5} + 2} + e_{\alpha_{5} +
4}e_{\alpha_{3} + \alpha_{5} + 3} +
\ldots + e_{\alpha_{5} + 2Deg_{z_{m_{1} + 1}}(LT(f))}e_{\alpha_{3} + \alpha_{5} + 2Deg_{z_{m_{1} + 1}}(LT(f)) - 1}$ \\

$\ldots$ \\

$z_{m_{1} + m_{2}} \mapsto e_{\alpha_{5} +
2(Deg_{z_{m_{1}+1}}(LT(f)) + \ldots + Deg_{z_{m_{1} + m_{2} -
1}}(LT(f))) + 2} + e_{\alpha_{5} + 2( Deg_{z_{m_{1} + 1}}(LT(f)) +
\ldots + Deg_{z_{m_{1} + m_{2} - 1}}(LT(f))) + 4}e_{\alpha_{3} +
\alpha_{5} + 2.(Deg_{z_{m_{1} + 1}}(LT(f)) + \ldots + Deg_{z_{m_{1}
+ m_{2} - 1}}(LT(f))) + 2(2 - m_{2}) + 1} + \ldots + e_{\alpha_{5} +
2( Deg_{z_{m_{1} + 1}}(LT(f)) + \ldots + Deg_{z_{m_{1} +
m_{2}}}(LT(f)))}e_{\alpha_{3} + \alpha_{5} + 2.(Deg_{z_{m_{1} +
1}}(LT(f)) + \ldots + Deg_{z_{m_{1} +
m_{2}}}(LT(f))) + 2(2 - (m_{2} + 1)) - 1}$ \\

$z_{m_{1} + m_{2} + 1} \mapsto e_{\alpha_{6} + 2}$ \\

$\ldots$ \\

$z_{m_{1} + m_{2} + m_{3}} \mapsto e_{\alpha_{6} + 2.m_{3}}$
\end{center}

It is easy to see that $dom(\phi(\gamma_{i}f_{i}u_{i})) =
\beta\gamma_{i}e_{1}.\cdots.e_{\alpha_{7}}e_{2}.\cdots.e_{\alpha_{8}}
$ for some $\beta \in F$ different from 0.

We are done if we prove $dom(\phi(\gamma_{i}f_{i}u_{i})) =
dom(\phi(f))$ because, in this case, we have $f = 0$ which is an
absurd.

Let $u_{j} \in \{u_{1},\ldots,u_{n}\}$ such that $u_{j} < u_{i}$. By
Lemma 7.1 of \cite{Fonseca}, none of the summands of
$dom(\phi(f_{j}u_{j}))$ contains
$\{e_{1},\ldots,e_{\alpha_{7}},e_{2},\ldots,e_{\alpha_{8}}\}$ in its
support. So we have $dom(\phi(f)) = dom(\phi(\gamma_{i}f_{i}u_{i}))$
and we are done.
\end{proof}

\begin{corolario}\label{impinf}
The set $pPol-SS(a_{1},\ldots,a_{l},b_{1},\ldots,b_{m})$ is a basis for $Multifree(a_{1},\ldots,a_{l},b_{1},\ldots,b_{m})$.
\end{corolario}

By Corollary \ref{impinf} and Lemma \ref{combssstar} we have the next.

\begin{teo*}
Let $n_1,n_2\in\N$, then \[c_{(n_1,n_2)}(E_{\infty},V)=c_{(n_1,n_2)}^*(SS).\]
\end{teo*}

\subsection{$E_{k^{*}}$}

\

\begin{def.}
Let us denote by $pPol-SS1(a_{1},\ldots,a_{l},b_{1},\ldots,b_{m})$
the set of polynomials of type $fu \in
Multifree(a_{1},\ldots,a_{l},b_{1},\ldots,b_{m})$, where \text{\rm
$f$ is a monomial $p$-polynomial with coefficient 1, $u_i\in
SS1$.}\end{def.}

We observe that if $fu,f'u' \in
Multifree(a_{1},\ldots,a_{l},b_{1},\ldots,b_{m})$ are such that $u =
u'$, then $f = f'$.

In light of of Proposition \ref{proposicao 9.2}, we have
$pPol-SS1(a_{1},\ldots,a_{l},b_{1},\ldots,b_{m})$ is a generating set for the vector space
$Multifree(a_{1},\ldots,a_{l},b_{1},\ldots,b_{m})$. We shall prove $pPol-SS1(a_{1},\ldots,a_{l},b_{1},\ldots,b_{m})$ is a linearly independent set.

\begin{prop}
The set $pPol-SS1(a_{1},\ldots,a_{l},b_{1},\ldots,b_{m})$ is linearly independent.
\end{prop}
\begin{proof}
Set $l=l_1+l_2+l_3$ and $m=m_1+m_2+m_3$.
Let us suppose $f\in pPol-SS1(a_{1},\ldots,a_{l},b_{1},\ldots,b_{m})$ such that $f = \sum_{i=1}^{n}\gamma_{i}f_{i}u_{i} =
0$, where $\gamma_{i} \neq 0$ for some $i \in \{1,\ldots,n\}$.
We shall denote by $LT(f)$ the biggest term of $\{u_{1},\ldots,u_{n}\}$. We shall suppose, without loss of generality, $u_{i} = LT(f)$.

Due to Corollary \ref{cor1} there exist
$\lambda_{1},\ldots,\lambda_{l} \in F$ such that
$f_{i}(\lambda_{1}1_{E},\ldots,\lambda_{l}1_{E}) \neq 0$.

We shall denote by $\alpha = k + 2(Deg_{y_{1}}(LT(f)) + \ldots +
Deg_{y_{l_{1}}}(LT(f))), \alpha_{1} = \alpha + 2(Deg_{y_{l_{1} +
1}}(LT(f)) + \ldots + Deg_{y_{l_{1} + l_{2}}}(LT(f))), \alpha_{2} =
\alpha_{1} + l_{3}, \alpha_{3} =  Deg_{z_{1}}(LT(f)) + \ldots +
Deg_{z_{m_{1}}}(LT(f))$.

Let us consider the following graded homomorphism $\phi: F\langle
y_{1},\ldots,y_{l_{1} + l_{2} + l_{3}}, z_{1}, \ldots, z_{m_{1} +
m_{2} + m_{3}} \rangle \rightarrow E$:

\begin{center}
$\phi: F\langle y_{1},\ldots,y_{l_{1}+l_{2}+l_{3}},
z_{1},\ldots,z_{m_{1}+m_{2}+m_{3}} \rangle \rightarrow E$ \\
$y_{1} \mapsto \lambda_{1}1_{E} + e_{k+1}e_{k+2} + \ldots + e_{k +
2(Deg_{y_{1}}(LT(f))) - 1}e_{k + 2(Deg_{y_{1}}(LT(f)))}$\\
$\ldots$ \\
$y_{l_{1}} \mapsto \lambda_{l_{1}}1_{E} + e_{k +
2(Deg_{y_{1}}(LT(f)) + \ldots + Deg_{y_{l_{1} - 1}}(LT(f))) + 1}e_{k
+ 2(Deg_{y_{1}}(LT(f)) + \ldots + Deg_{y_{l_{1} - 1}}(LT(f))) + 2} +
\ldots + e_{k + 2(Deg_{y_{1}}(LT(f)) + \ldots +
Deg_{y_{l_{1}}}(LT(f))) - 1}e_{k + 2(Deg_{y_{1}}(LT(f)) + \ldots +
Deg_{y_{l_{1}}}(LT(f)))}$ \\
$y_{l_{1} + 1} \mapsto \lambda_{l_{1} + 1}1_{E} + e_{\alpha + 1} +
e_{\alpha + 2}e_{\alpha + 3}
+ \ldots + e_{\alpha + 2Deg_{y_{l_{1} + 1}}(LT(f))}e_{\alpha + 2Deg_{y_{l_{1} + 1}}(LT(f)) + 1}$ \\
$\ldots$ \\
$y_{l_{1} + l_{2}} \mapsto \lambda_{l_{1} + l_{2}}1_{E} + e_{\alpha
+ 2(Deg_{y_{l_{1} + 1}}(LT(f)) + \ldots + Deg_{y_{l_{1} + l_{2} -
1}}(LT(f))) + l_{2}} + e_{\alpha + 2(Deg_{y_{l_{1}+1}}(LT(f)) +
\ldots + Deg_{y_{l_{1} + l_{2} - 1}}(LT(f))) + l_{2} + 1}e_{\alpha +
2(Deg_{y_{l_{1}+1}}(LT(f)) + \ldots + Deg_{y_{l_{1} + l_{2} -
1}}(LT(f))) + l_{2} + 2} + \ldots + e_{\alpha + 2(Deg_{y_{l_{1} +
1}}(LT(f)) + \ldots + Deg_{y_{l_{1} + l_{2}}}(LT(f))) + l_{2} -
1}e_{\alpha + 2(Deg_{y_{l_{1} + 1}}(LT(f)) + \ldots
+ Deg_{y_{l_{1} + l_{2}}}(LT(f))) + l_{2}} $\\
$\ldots$\\
$y_{l_{1} + l_{2} + l_{3}} \mapsto \lambda_{l_{1} + l_{2} + l_{3}}1_{E} + e_{\alpha_{1} + l_{3}}$\\
$z_{1} \mapsto e_{1}e_{\alpha_{2}} + \ldots +
e_{Deg_{z_{1}}(LT(f))}e_{\alpha_{2} + Deg_{z_{1}}(LT(f))}$\\
$\ldots$ \\
$z_{m_{1}} \mapsto e_{Deg_{z_{1}}(LT(f)) + \ldots + Deg_{z_{m_{1} -
1}}(LT(f)) + 1 }e_{\alpha_{2} + Deg_{z_{1}}(LT(f)) + \ldots +
Deg_{z_{m_{1} - 1}}(LT(f)) + 1} + \ldots + e_{Deg_{z_{1}}(LT(f)) +
\ldots + Deg_{z_{m_{1}}}(LT(f))}e_{\alpha_{2} + Deg_{z_{1}}(LT(f)) +
\ldots + Deg_{z_{m_{1}}}(LT(f))}$\\
$z_{m_{1} + 1} \mapsto e_{\alpha_{3} + 1} + e_{\alpha_{3} +
2}e_{\alpha_{2} + \alpha_{3} + 1} + \ldots + e_{\alpha_{3} +
Deg_{z_{m_{1} + 1}}(LT(f))}e_{\alpha_{2} + \alpha_{3} +
Deg_{z_{m_{1}
+ 1}}(LT(f)) - 1}$ \\
$\ldots$ \\
$z_{m_{1} + m_{2}} \mapsto e_{\alpha_{3} + Deg_{z_{m_{1} +
1}}(LT(f)) + \ldots + Deg_{z_{m_{1} + m_{2} - 1}}(LT(f)) + 1} +
e_{\alpha_{3} + Deg_{z_{m_{1} + 1}}(LT(f)) + \ldots + Deg_{z_{m_{1}
+ m_{2} - 1}}(LT(f)) + 2}e_{\alpha_{2} + \alpha_{3} + Deg_{z_{m_{1}
+ 1}}(LT(f)) + \ldots + Deg_{z_{m_{1} + m_{2} - 1}}(LT(f)) - m_{2} +
2} + \ldots + e_{\alpha_{3} + Deg_{z_{m_{1} + 1}}(LT(f)) + \ldots +
Deg_{z_{m_{1} + m_{2}}}(LT(f))}e_{\alpha_{2} + \alpha_{3} +
Deg_{z_{m_{1} + 1}}(LT(f))
+ \ldots + Deg_{z_{m_{1} + m_{2}}}(LT(f)) - m_{2}}$ \\
$z_{m_{1} + m_{2} + 1} \mapsto e_{\alpha_{3} + Deg_{z_{m_{1} +
1}}(LT(f)) + \ldots + Deg_{z_{m_{1} + m_{2}}}(LT(f)) + 1}$ \\
$\ldots$ \\
$z_{m_{1} + m_{2} + m_{3}} \mapsto e_{\alpha_{3} + Deg_{z_{m_{1} +
1}}(LT(f)) + \ldots + Deg_{z_{m_{1} + m_{2}}}(LT(f)) + m_{3}}$
\end{center}

It is not difficult to see that
\[dom(\phi(\gamma_{i}f_{i}u_{i}))\]\[ = \beta\gamma_{i}e_{1}\cdots
e_{\alpha_{3} + Deg_{z_{m_{1} + 1}}(LT(f)) + \ldots + Deg_{z_{m_{1}
+ m_{2}}}(LT(f)) + m_{3}}e_{k+1}\cdots e_{\alpha_{2} + \alpha_{3} +
Deg_{z_{m_{1} + 1}}(LT(f)) + \ldots + Deg_{z_{m_{1} + m_{2}}}(LT(f))
- m_{2}}\] for some non-zero $\beta \in F$.

Notice that we are done if we prove $dom(\phi(\gamma_{i}f_{i}u_{i}))
= dom(\phi(f)) = 0$.

Let us consider $u_{j} \in \{u_{1},\ldots,u_{n}\}$. By Lemma 7.1 of
\cite{Fonseca}, none of the summands of
$dom(\phi(\gamma_{j}f_{j}u_{j}))$ contains
$supp(dom(\phi(\gamma_{i}f_{i}u_{i})))$, then
$dom(\phi(\gamma_{i}f_{i}u_{i})) = dom(\phi(f))$ and we are done.
\end{proof}

\begin{corolario}\label{impkstar}
The set $pPol-SS1(a_{1},\ldots,a_{l},b_{1},\ldots,b_{m})$ is a basis for $Multifree(a_{1},\ldots,a_{l},b_{1},\ldots,b_{m})$.
\end{corolario}

By Corollary \ref{impkstar} and Lemma \ref{combssstar} we have the next.

\begin{teo*}
Let $n_1,n_2,k\in\N$, then \[c_{(n_1,n_2)}(E_{k^*},V)=c_{(n_1,n_2)}^*(SS1).\]
\end{teo*}

\subsection{$E_{k}, k \geq 1$}

\

\begin{def.}
Let us denote by $pPol-SS3(a_{1},\ldots,a_{l},b_{1},\ldots,b_{m})$
the set of polynomials of type $fu \in
Multifree(a_{1},\ldots,a_{l},b_{1},\ldots,b_{m})$, where \text{\rm
$f$ is a monomial $p$-polynomial with coefficient 1, $u_i\in
SS3$.}\end{def.}

As above, if $fu,f'u' \in
Multifree(a_{1},\ldots,a_{l},b_{1},\ldots,b_{m})$ are such that $u =
u'$, then $f = f'$.

Due to Proposition \ref{proposicao 10.16}, we have
$pPol-SS3(a_{1},\ldots,a_{l},b_{1},\ldots,b_{m})$ is a generating set for the vector space
$Multifree(a_{1},\ldots,a_{l},b_{1},\ldots,b_{m})$. In what follows we shall prove $pPol-SS3(a_{1},\ldots,a_{l},b_{1},\ldots,b_{m})$ is a linearly independent set.

In what follows we use an idea of proof used by Fonseca in
\cite{Fonseca} (see Theorem 10.17 of \cite{Fonseca}).

\begin{prop}
The set $pPol-SS3(a_{1},\ldots,a_{l},b_{1},\ldots,b_{m})$ is linearly independent.
\end{prop}
\begin{proof}
Set $l=l_1+l_2+l_3$ and $m=m_1+m_2+m_3$. Suppose that there exists
$f\in pPol-SS3(a_{1},\ldots,a_{l},b_{1},\ldots,b_{m})$ such that $f
= \sum_{i = 1}^{n}\gamma_{i}f_{i}u_{i} = 0$, with $\lambda_{i} \neq
0$ for some $i \in \{1,\ldots,n\}$. Moreover, we shall suppose,
without loss of generality, $u_{i} = LT(f)$. We shall denote by
$\alpha_{1} = k + 2(Deg_{y_{1}}(LT(f)) + \ldots +
Deg_{y_{l_{1}}}(LT(f))), \alpha_{2} = \alpha_{1} + 2(Deg_{y_{l_{1} +
1}}(LT(f)) + \ldots + Deg_{y_{l_{1} + l_{2} - 1}}(LT(f))) - 2l_{2},
\alpha_{3} = \alpha_{2} + Deg_{z_{1}}(LT(f)) + \ldots +
Deg_{z_{m_{1}}}(LT(f)), \alpha_{4} = l_{2} + l_{3} +
Deg_{z_{1}}(LT(f)) + \ldots + Deg_{z_{m_{1}}}(LT(f)), \alpha_{5} =
\alpha_{3} + Deg_{z_{m_{1}+1}}(LT(f)) + \ldots + Deg_{z_{m_{1} +
m_{2}}}(LT(f)) + m_{3}, \alpha_{6} = \alpha_{4} +
Deg_{z_{m_{1}+1}}(LT(f)) + \ldots + Deg_{z_{m_{1} + m_{2}}}(LT(f)) -
m_{2}$.

Because $f_{i}(y_{1},\ldots,y_{l})$ is a non-zero $p$-polynomial,
due to Corollary \ref{cor1} there exist scalars
$\lambda_{1},\ldots,\lambda_{l} \in F$ such that
$f_{i}(\lambda_{1}1_{E},\ldots,\lambda_{l}1_{E}) \neq 0$.

We have to consider three cases.

\begin{description}
\item Case 1: $deg_{Z}(beg(LT(f))) + deg_{Y}(\psi(LT(f))) \leq k$;
\item Case 2: $deg_{Z}(beg(LT(f))) + deg_{Y}(\psi(LT(f))) = k + 1$ and none of the elements of $\{u_{1},\ldots,u_{n}\}$ is a bad term;
\item Case 3: $deg_{Z}(beg(LT(f))) + deg_{Y}(\psi(LT(f))) = k + 1$ and there exists a term $u_{j} \in \{u_{1},\ldots,u_{n}\}$ such that $\lambda_{j} \neq
0$ and $u_{j}$ is a bad term.
\end{description}

Let us study Case 1 first. Let us consider the following graded homomorphism:

\begin{center}
$\phi: F\langle y_{1},\ldots,y_{l_{1}+l_{2}+l_{3}},z_{1},\ldots,z_{m_{1}+m_{2}+m_{3}}\rangle \rightarrow E$\\

$y_{1} \mapsto \lambda_{1}1_{E} + e_{k+1}e_{k+2} + \ldots + e_{k +
2Deg_{y_{1}}(LT(f))
- 1}e_{k + 2Deg_{y_{1}}(LT(f))}$ \\
$\ldots$ \\
$y_{l_{1}} \mapsto \lambda_{l_{1}}1_{E} + e_{k +
2(Deg_{y_{1}}(LT(f)) + \ldots + Deg_{y_{l_{1} - 1}}(LT(f))) + 1}e_{k
+ 2(Deg_{y_{1}}(LT(f)) + \ldots + Deg_{y_{l_{1} - 1}}(LT(f))) + 2} +
\ldots + e_{k + 2(Deg_{y_{1}}(LT(f)) + \ldots +
Deg_{y_{l_{1}}}(LT(f))) - 1}e_{k +
2(Deg_{y_{1}}(LT(f)) + \ldots + Deg_{y_{l_{1}}}(LT(f)))}$ \\
$y_{l_{1} + 1} \mapsto \lambda_{l_{1} + 1}1_{E} + e_{1} +
e_{\alpha_{1} + 1}e_{\alpha_{1} + 2} + \ldots + e_{\alpha_{1} +
2Deg_{y_{l_{1} + 1}}(LT(f)) - 3}e_{\alpha_{1} +
2Deg_{y_{l_{1} + 1}}(LT(f)) - 2}$ \\
$\ldots$ \\
$y_{l_{1} + l_{2}} \mapsto \lambda_{l_{1} + l_{2}}1_{E} + e_{l_{2}}
+ e_{\alpha_{1} + 2(Deg_{y_{l_{1} + 1}}(LT(f)) + \ldots +
Deg_{y_{l_{1} + l_{2} - 1}}(LT(f))) + 2(1 - l_{2}) + 1}e_{\alpha_{1}
+ 2(Deg_{y_{l_{1} + 1}}(LT(f)) + \ldots + Deg_{y_{l_{1} + l_{2} -
1}}(LT(f))) + 2(1 - l_{2}) + 2} + \ldots + e_{\alpha_{1} +
2(Deg_{y_{l_{1} + 1}}(LT(f)) + \ldots + Deg_{y_{l_{1} +
l_{2}}}(LT(f))) - 2l_{2} - 1}e_{\alpha_{1} + 2(Deg_{y_{l_{1} +
1}}(LT(f)) + \ldots +
Deg_{y_{l_{1} + l_{2} - 1}}(LT(f))) - 2l_{2}}$ \\
$y_{l_{1} + l_{2} + 1} \mapsto \lambda_{l_{1} + l_{2} + 1}1_{E} + e_{l_{2} + 1}$ \\
$\ldots$ \\
$y_{l_{1} + l_{2} + l_{3}} \mapsto \lambda_{l_{1} + l_{2} + l_{3}}1_{E} + e_{l_{2} + l_{3}}$\\
$z_{1} \mapsto e_{\alpha_{2} + 1}e_{l_{2} + l_{3} + 1} + \ldots +
e_{\alpha_{2} + Deg_{z_{1}}(LT(f))}e_{l_{2} + l_{3} +
Deg_{z_{1}}(LT(f))}$ \\
$\ldots$ \\
$z_{m_{1}} \mapsto e_{\alpha_{2} + Deg_{z_{1}}(LT(f)) + \ldots +
Deg_{z_{m_{1} - 1}}(LT(f)) + 1}e_{l_{2} + l_{3} + Deg_{z_{1}}(LT(f))
+ \ldots + Deg_{z_{m_{1} - 1}}(LT(f)) + 1} + \ldots + e_{\alpha_{2}
+ Deg_{z_{1}}(LT(f)) + \ldots + Deg_{z_{m_{1}}}(LT(f))}e_{l_{2} +
l_{3} + Deg_{z_{1}}(LT(f)) + \ldots + Deg_{z_{m_{1}}}(LT(f))}$ \\
$\ldots$\\
$z_{m_{1} + 1} \mapsto e_{\alpha_{3}} + e_{\alpha_{3} + 1}e_{\alpha_{4} + 1} + \ldots + e_{\alpha_{3} + Deg_{z_{m_{1} + 1}}(LT(f))}e_{\alpha_{4} + Deg_{z_{m_{1} + 1}}(LT(f)) - 1}$ \\
$\ldots$ \\
$z_{m_{1} + m_{2}} \mapsto e_{\alpha_{3} + Deg_{z_{m_{1}+1}}(LT(f))
+ \ldots + Deg_{z_{m_{1} + m_{2} - 1}}(LT(f))} + e_{\alpha_{4} +
Deg_{z_{m_{1}+1}}(LT(f)) + \ldots + Deg_{z_{m_{1} + m_{2} -
1}}(LT(f)) - m_{2} + 2}e_{\alpha_{3} + Deg_{z_{m_{1}+1}}(LT(f)) +
\ldots + Deg_{z_{m_{1} + m_{2} - 1}}(LT(f)) + 1} + \ldots +
e_{\alpha_{4} + Deg_{z_{m_{1}+1}}(LT(f)) + \ldots + Deg_{z_{m_{1} + m_{2}}}(LT(f)) - m_{2}}e_{\alpha_{3} + Deg_{z_{m_{1}+1}}(LT(f)) + \ldots + Deg_{z_{m_{1} + m_{2}}}(LT(f))}$ \\
$z_{m_{1} + m_{2} + 1} \mapsto e_{\alpha_{3} + Deg_{z_{m_{1}+1}}(LT(f)) + \ldots + Deg_{z_{m_{1} + m_{2}}}(LT(f)) + 1}$ \\
$z_{m_{1} + m_{2} + m_{3}} \mapsto e_{\alpha_{3} +
Deg_{z_{m_{1}+1}}(LT(f)) + \ldots + Deg_{z_{m_{1} + m_{2}}}(LT(f)) +
m_{3}}$
\end{center}

In this case we have $dom(\phi(\gamma_{i}f_{i}u_{i})) =
\beta.\gamma_{i}
e_{1}.\cdots.e_{\alpha_{6}}.e_{k+1}.\cdots.e_{\alpha_{5}}$ for some
$\beta \in F - \{0\}$. By Lemma 7.1 of \cite{Fonseca}, we obtain
$dom(\phi(\gamma_{i}f_{i}u_{i})) = \beta.\gamma_{i}
e_{1}.\cdots.e_{\alpha_{6}}.e_{k+1}.\cdots.e_{\alpha_{5}} = dom(f) =
0$ which is a contradiction.

\

Now we study Case 2. Let us consider the following graded homomorphism:

\begin{center}
$\phi: F\langle y_{1},\ldots,y_{l_{1}+l_{2}+l_{3}},z_{1},\ldots,z_{m_{1}+m_{2}+m_{3}}\rangle \rightarrow E$\\

$y_{1} \mapsto \lambda_{1}1_{E} + e_{k+1}e_{k+2} + \ldots + e_{k +
2Deg_{y_{1}}(LT(f))
- 1}e_{k + 2Deg_{y_{1}}(LT(f))}$ \\
$\ldots$ \\
$y_{l_{1}} \mapsto \lambda_{l_{1}}1_{E} + e_{k +
2(Deg_{y_{1}}(LT(f)) + \ldots + Deg_{y_{l_{1} - 1}}(LT(f))) + 1}e_{k
+ 2(Deg_{y_{1}}(LT(f)) + \ldots + Deg_{y_{l_{1} - 1}}(LT(f))) + 2} +
\ldots + e_{k + 2(Deg_{y_{1}}(LT(f)) + \ldots +
Deg_{y_{l_{1}}}(LT(f))) - 1}e_{k +
2(Deg_{y_{1}}(LT(f)) + \ldots + Deg_{y_{l_{1}}}(LT(f)))}$ \\
$y_{l_{1} + 1} \mapsto \lambda_{l_{1}+1}1_{E} + e_{1} +
e_{\alpha_{1} + 1}e_{\alpha_{1} + 2} + \ldots + e_{\alpha_{1} +
2Deg_{y_{l_{1} + 1}}(LT(f)) - 3}e_{\alpha_{1} +
2Deg_{y_{l_{1} + 1}}(LT(f)) - 2}$ \\
$\ldots$ \\
$y_{l_{1} + l_{2}} \mapsto \lambda_{l_{1}+l_{2}}1_{E} + e_{l_{2}} +
e_{\alpha_{1} + 2(Deg_{y_{l_{1} + 1}}(LT(f)) + \ldots +
Deg_{y_{l_{1} + l_{2} - 1}}(LT(f))) + 2(1 - l_{2}) + 1}e_{\alpha_{1}
+ 2(Deg_{y_{l_{1} + 1}}(LT(f)) + \ldots + Deg_{y_{l_{1} + l_{2} -
1}}(LT(f))) + 2(1 - l_{2}) + 2} + \ldots + e_{\alpha_{1} +
2(Deg_{y_{l_{1} + 1}}(LT(f)) + \ldots + Deg_{y_{l_{1} +
l_{2}}}(LT(f))) - 2l_{2} - 1}e_{\alpha_{1} + 2(Deg_{y_{l_{1} +
1}}(LT(f)) + \ldots +
Deg_{y_{l_{1} + l_{2} - 1}}(LT(f))) - 2l_{2}}$ \\
$y_{l_{1} + l_{2} + 1} \mapsto \lambda_{l_{1}+l_{2}+1}1_{E} + e_{l_{2} + 1}$ \\
$\ldots$ \\
$y_{l_{1} + l_{2} + l_{3}} \mapsto \lambda_{l_{1}+l_{2}+l_{3}}1_{E} + e_{l_{2} + l_{3}}$\\
$z_{1} \mapsto e_{\alpha_{2} + 1} + \ldots + e_{\alpha_{2} +
Deg_{z_{1}}(LT(f))}e_{l_{2} + l_{3} +
Deg_{z_{1}}(LT(f)) - 1}$ \\
$\ldots$ \\
$z_{m_{1}} \mapsto e_{\alpha_{2} + Deg_{z_{1}}(LT(f)) + \ldots +
Deg_{z_{m_{1} - 1}}(LT(f))}e_{l_{2} + l_{3} + Deg_{z_{1}}(LT(f)) +
\ldots + Deg_{z_{m_{1} - 1}}(LT(f))} + \ldots + e_{\alpha_{2} +
Deg_{z_{1}}(LT(f)) + \ldots + Deg_{z_{m_{1}}}(LT(f))}e_{l_{2} +
l_{3} + Deg_{z_{1}}(LT(f)) + \ldots + Deg_{z_{m_{1}}}(LT(f)) - 1}$ \\
$\ldots$\\
$z_{m_{1} + 1} \mapsto e_{\alpha_{3}} + e_{\alpha_{3} + 1}e_{\alpha_{4}} + \ldots + e_{\alpha_{3} + Deg_{z_{m_{1} + 1}}(LT(f))}e_{\alpha_{4} + Deg_{z_{m_{1} + 1}}(LT(f)) - 2}$ \\
$\ldots$ \\
$z_{m_{1} + m_{2}} \mapsto e_{\alpha_{3} + Deg_{z_{m_{1}+1}}(LT(f))
+ \ldots + Deg_{z_{m_{1} + m_{2} - 1}}(LT(f))} + e_{\alpha_{4} +
Deg_{z_{m_{1}+1}}(LT(f)) + \ldots + Deg_{z_{m_{1} + m_{2} -
1}}(LT(f)) - m_{2} + 1}e_{\alpha_{3} + Deg_{z_{m_{1}+1}}(LT(f)) +
\ldots + Deg_{z_{m_{1} + m_{2} - 1}}(LT(f)) + 1} + \ldots +
e_{\alpha_{4} + Deg_{z_{m_{1}+1}}(LT(f)) + \ldots + Deg_{z_{m_{1} + m_{2}}}(LT(f)) - m_{2} - 1}e_{\alpha_{3} + Deg_{z_{m_{1}+1}}(LT(f)) + \ldots + Deg_{z_{m_{1} + m_{2}}}(LT(f))}$ \\
$z_{m_{1} + m_{2} + 1} \mapsto e_{\alpha_{3} + Deg_{z_{m_{1}+1}}(LT(f)) + \ldots + Deg_{z_{m_{1} + m_{2}}}(LT(f)) + 1}$ \\
$z_{m_{1} + m_{2} + m_{3}} \mapsto e_{\alpha_{3} +
Deg_{z_{m_{1}+1}}(LT(f)) + \ldots + Deg_{z_{m_{1} + m_{2}}}(LT(f)) +
m_{3}}$
\end{center}

Notice that $dom(\phi(\gamma_{i}f_{i}u_{i})) = \beta.\gamma_{i}
e_{1}.\cdots.e_{\alpha_{6} - 1}.e_{k+1}.\cdots.e_{\alpha_{5}}$ for
some $\beta \in F - \{0\}$. By Lemma 7.2 of \cite{Fonseca}, we
obtain $dom(\phi(\gamma_{i}f_{i}u_{i})) = dom(f) = 0$ which is a
contradiction.

\

Finally, we deal with Case 3. We suppose that $u_{k} = LBT(f)$.
Notice that for Case 3, we will have $deg_{Z}(beg(LBT(f))) +
deg_{Y}(\psi(LBT(f))) \leq k$. Keeping in mind Case 2, we have
$z_{1} = pr(z)(LT(f))$ when $deg_{Z}(beg(LT(f))) +
deg_{Y}(\psi(LT(f))) = k + 1$.

We shall construct a graded homomorphism assuming, without loss of
generality, the next facts:
\begin{description}
\item $Yyn(LBT(f)) = \{y_{1},\ldots,y_{n_{1}}\} (\mbox{with}
\ |Yyy(LBT(f))| \geq 2)$;

\item $Yyn(LBT(f)) =
\{y_{n_{1}+1},\ldots,y_{n_{1}+n_{2}}\} (\mbox{with} \ |Yyn(LBT(f))|
\geq 2)$;

\item $Yny(LBT(f)) =
\{y_{n_{1}+n_{2}+1},\ldots,y_{n_{1}+n_{2}+n_{3}}\}$;

\item $Zyn(LBT(f)) = \{z_{2},\ldots,z_{m_{1}}\} (\mbox{with}
\ |Zyn(LBT(f))| \geq 2)$;

\item $Zyy(LBT(f)) = \{z_{1},z_{m_{1} + 1},\ldots,z_{m_{1}+m_{2}}\}
(\mbox{with} \ |Zyy(LBT(f))| \geq 2)$;

\item $Zny(LBT(f)) =
\{z_{m_{1}+m_{2}+1},\ldots,z_{m_{1}+m_{2}+m_{3}}\}$.
\end{description}

There exists scalars $\lambda_{1},\ldots,\lambda_{l} \in F$ such
that $f_{k}(\lambda_{1}1_{E},\ldots,\lambda_{l}1_{E}) \neq 0$

We consider the following graded homomorphism:

\begin{center}
$\phi: F\langle y_{1},\ldots,y_{l_{1}+l_{2}+l_{3}},z_{1},\ldots,z_{m_{1}+m_{2}+m_{3}}\rangle \rightarrow E$\\

$y_{1} \mapsto \lambda_{1}1_{E} + e_{k+1}e_{k+2} + \ldots + e_{k +
2Deg_{y_{1}}(LBT(f))
- 1}e_{k + 2Deg_{y_{1}}(LBT(f))}$ \\
$\ldots$ \\
$y_{l_{1}} \mapsto \lambda_{l_{1}}1_{E} +  e_{k +
2(Deg_{y_{1}}(LBT(f)) + \ldots + Deg_{y_{l_{1} - 1}}(LBT(f))) +
1}e_{k + 2(Deg_{y_{1}}(LBT(f)) + \ldots + Deg_{y_{l_{1} -
1}}(LBT(f))) + 2} + \ldots + e_{k + 2(Deg_{y_{1}}(LBT(f)) + \ldots +
Deg_{y_{l_{1}}}(LBT(f))) - 1}e_{k +
2(Deg_{y_{1}}(LBT(f)) + \ldots + Deg_{y_{l_{1}}}(LBT(f)))}$ \\
$y_{l_{1} + 1} \mapsto \lambda_{l_{1} + 1}1_{E} + e_{1} +
e_{\alpha_{1} + 1}e_{\alpha_{1} + 2} + \ldots + e_{\alpha_{1} +
2Deg_{y_{l_{1} + 1}}(LBT(f)) - 3}e_{\alpha_{1} +
2Deg_{y_{l_{1} + 1}}(LBT(f)) - 2}$ \\
$\ldots$ \\
$y_{l_{1} + l_{2}} \mapsto \lambda_{l_{1} + l_{2}}1_{E} + e_{l_{2}}
+ e_{\alpha_{1} + 2(Deg_{y_{l_{1} + 1}}(LBT(f)) + \ldots +
Deg_{y_{l_{1} + l_{2} - 1}}(LBT(f))) + 2(1 - l_{2}) +
1}e_{\alpha_{1} + 2(Deg_{y_{l_{1} + 1}}(LBT(f)) + \ldots +
Deg_{y_{l_{1} + l_{2} - 1}}(LBT(f))) + 2(1 - l_{2}) + 2} + \ldots +
e_{\alpha_{1} + 2(Deg_{y_{l_{1} + 1}}(LBT(f)) + \ldots +
Deg_{y_{l_{1} + l_{2}}}(LBT(f))) - 2l_{2} - 1}e_{\alpha_{1} +
2(Deg_{y_{l_{1} + 1}}(LBT(f)) + \ldots +
Deg_{y_{l_{1} + l_{2} - 1}}(LBT(f))) - 2l_{2}}$ \\
$y_{l_{1} + l_{2} + 1} \mapsto \lambda_{l_{1} + l_{2} + 1}1_{E} + e_{l_{2} + 1}$ \\
$\ldots$ \\
$y_{l_{1} + l_{2} + l_{3}} \mapsto \lambda_{l_{1} + l_{2} + l_{3}}1_{E} + e_{l_{2} + l_{3}}$\\
$z_{2} \mapsto e_{\alpha_{2} + 1}e_{l_{2} + l_{3} + 1} + \ldots +
e_{\alpha_{2} + Deg_{z_{1}}(LBT(f))}e_{l_{2} + l_{3} +
Deg_{z_{1}}(LBT(f))}$ \\
$\ldots$ \\
$z_{m_{1}} \mapsto e_{\alpha_{2} + Deg_{z_{1}}(LBT(f)) + \ldots +
Deg_{z_{m_{1} - 1}}(LBT(f)) + 1}e_{l_{2} + l_{3} +
Deg_{z_{1}}(LBT(f)) + \ldots + Deg_{z_{m_{1} - 1}}(LBT(f)) + 1} +
\ldots + e_{\alpha_{2} + Deg_{z_{1}}(LBT(f)) + \ldots +
Deg_{z_{m_{1}}}(LBT(f))}e_{l_{2} +
l_{3} + Deg_{z_{1}}(LBT(f)) + \ldots + Deg_{z_{m_{1}}}(LBT(f))}$ \\
$\ldots$\\
$z_{1} \mapsto e_{\alpha_{3}} + e_{\alpha_{3} + 1}e_{\alpha_{4} + 1} + \ldots + e_{\alpha_{3} + Deg_{z_{m_{1} + 1}}(LBT(f))}e_{\alpha_{4} + Deg_{z_{m_{1} + 1}}(LBT(f)) - 1}$ \\
$\ldots$ \\
$z_{m_{1} + m_{2}} \mapsto e_{\alpha_{3} + Deg_{z_{m_{1}+1}}(LBT(f))
+ \ldots + Deg_{z_{m_{1} + m_{2} - 1}}(LBT(f))} + e_{\alpha_{4} +
Deg_{z_{m_{1}+1}}(LBT(f)) + \ldots + Deg_{z_{m_{1} + m_{2} -
1}}(LBT(f)) - m_{2} + 2}e_{\alpha_{3} + Deg_{z_{m_{1}+1}}(LBT(f)) +
\ldots + Deg_{z_{m_{1} + m_{2} - 1}}(LBT(f)) + 1} + \ldots +
e_{\alpha_{4} + Deg_{z_{m_{1}+1}}(LBT(f)) + \ldots + Deg_{z_{m_{1} + m_{2}}}(LBT(f)) - m_{2}}e_{\alpha_{3} + Deg_{z_{m_{1}+1}}(LBT(f)) + \ldots + Deg_{z_{m_{1} + m_{2}}}(LBT(f))}$ \\
$z_{m_{1} + m_{2} + 1} \mapsto e_{\alpha_{3} + Deg_{z_{m_{1}+1}}(LBT(f)) + \ldots + Deg_{z_{m_{1} + m_{2}}}(LBT(f)) + 1}$ \\
$z_{m_{1} + m_{2} + m_{3}} \mapsto e_{\alpha_{3} +
Deg_{z_{m_{1}+1}}(LBT(f)) + \ldots + Deg_{z_{m_{1} + m_{2}}}(LBT(f))
+ m_{3}}$
\end{center}

Notice that $dom(\phi(\gamma_{k}f_{k}u_{k})) = \beta.\gamma_{k}
e_{1}.\ldots.e_{\alpha_{6}}e_{k+1}\ldots e_{\alpha_{5}}$ for some
$\beta \in F - \{0\}$. By Lemma 7.3 of \cite{Fonseca}, we get $0 =
dom(\phi(f)) = dom(\phi(\gamma_{k}f_{k}u_{k}))$, which completes the
proof.
\end{proof}

\begin{corolario}\label{impk}
The set $pPol-SS3(a_{1},\ldots,a_{l},b_{1},\ldots,b_{m})$ is a basis for $Multifree(a_{1},\ldots,a_{l},b_{1},\ldots,b_{m})$.
\end{corolario}

By Corollary \ref{impk} and Lemma \ref{combssstar} we have the next.

\begin{teo*}
Let $n_1,n_2,k\in\N$, then \[c_{(n_1,n_2)}(E_{k},V)=c_{(n_1,n_2)}^*(SS3).\]
\end{teo*}

\section{Bounds for $\Z_2$-graded codimensions with respect to a generating subspace of $F\langle X\rangle$}
As mentioned above, if $W$ is a generating subspace of $F\langle X\rangle$, we are not able to recover the codimensions of $W(n)$ from the homogeneous ones. In this section we provide a lower and an upper bound for $c_{(n_1,n_2)}(E,W(n))$ for each homogeneous $\Z_2$-grading of $E$.

\subsection{$E_{can}$}

\

We start with the canonical $\mathbb{Z}_2$-grading on $E$.

\begin{lema}
The polynomials $SS0 \cap F \langle
y_{1},\ldots,y_{l},z_{1},\ldots,z_{m}\rangle$ are linearly independent modulo $T_{2}(E_{can})$.
\end{lema}
\begin{proof}
Let $f = \sum_{i = 1}^{n}\lambda_{i}M_{i} \equiv 0$, where the
$M_i$'s belong to $SS0 \cap F \langle
y_{1},\ldots,y_{l},z_{1},\ldots,z_{m}\rangle$ modulo
$T_{2}(E_{can})$. We shall assume without loss of generality that
each variable appears in each summand of $f$ with degree at least.
Suppose, by contradiction that there exists $M_{j} \in
\{M_{1},\ldots,M_{l}\}$, such that $\lambda_{j} \neq 0$. Moreover we
shall assume $M_{j} = LT(f)$. Let us define $\alpha =
2(Deg_{y_{1}}(LT(f)) + \ldots + Deg_{y_{n_{1}}}(LT(f)))$.

Let $n_1\leq l$ and $n_2\leq m$. We consider the graded homomorphism $\phi : F\langle
y_{1},\ldots,y_{n_{1}},z_{1},\ldots,z_{n_{2}}\rangle \rightarrow E_{can}$:

\begin{center}
$y_{1} \mapsto e_{1}e_{2} + \ldots + e_{2Deg_{y_{1}}(LT(f)) -
1}e_{2Deg_{y_{1}}(LT(f))}$ \\
$\ldots$ \\
$y_{n_{1}} \mapsto e_{2(Deg_{y_{1}}(LT(f)) + \ldots + Deg_{y_{n_{1}
- 1}}(LT(f))) + 1}e_{2(Deg_{y_{1}}(LT(f)) + \ldots + Deg_{y_{n_{1} -
1}}(LT(f))) + 2} + \ldots + e_{2(Deg_{y_{1}}(LT(f)) + \ldots +
Deg_{y_{n_{1}}}(LT(f))) - 1}e_{2(Deg_{y_{1}}(LT(f)) + \ldots +
Deg_{y_{n_{1}}}(LT(f)))}$ \\
$z_{1} \mapsto e_{\alpha + 1}$\\
$\ldots$ \\
$z_{n_{2}} \mapsto e_{\alpha + n_{2}}$
\end{center}
Notice that $\phi(LT(f)) = \lambda e_{1}.\cdots.e_{\alpha + n_{2}}$
for some $\lambda \in F - \{0\}$.

If there exists $M_{i} \in \{M_{1},\ldots,M_{l}\}$, such that $M_{i}
< LT(f)$ and $\lambda_{i} \neq 0$, then there exists $x \in Y$ such
that $Deg_{x} (M_{i}) < Deg_{x} (LT(f))$. Let us observe that
$supp(\phi(x))$ is not contained in any of the summand of
$dom(\phi(M_{i}))$. So the support of each summand of
$dom(\phi(M_{i}))$ does not contain $supp(\phi(LT(f)))$. Then
$dom(\phi(f)) = \phi(LT(f)) = \lambda.e_{1}\cdots e_{\alpha + n_{2}}
= 0$ which is a contradiction and we are done.
\end{proof}

Combining Lemma \ref{combss0} and Proposition \ref{oo} and in light of Lemma \ref{combssstar2} we have the following result.

\begin{teo}
For each $n\in\N$ and $n_1,n_2\in\N$ such that $n_1+n_2=n$ we have \[c_{(n_1,n_2)}(SS0)\leq c_{(n_1,n_2)}(E_{can},W(n))\leq
c_{(n_1,n_2)}^{\circ}(SS0).\]
\end{teo}

\subsection{$E_{\infty}$}

\

We shall consider now the case of $E_{\infty}$.

\begin{lema}\label{primeiro lema}
The polynomials $SS \cap F\langle
y_{1},\ldots,y_{l},z_{1},\ldots,z_{m}\rangle$ form a linearly independent set modulo $T_{2}(E_{\infty})$.
\end{lema}
\begin{proof}
Let $f = \sum\limits_{i = 1}^{n} \lambda_{i}M_{i} \equiv 0$ a linear
combination modulo $T_{2}(E_{\infty})$, of elements of $SS \cap
F\langle y_{1},\ldots,y_{l},z_{1},\ldots,z_{m}\rangle$. We may
assume each variable appears in each summand of $f$ with degree at
least 1. Suppose for the sake of contradiction that $M_{j} \in
\{M_{1},\ldots,M_{n}\}$, such that $\lambda_{j} \neq 0$.  We shall
also assume, without loss of generality that $M_{j} = LT(f)$. Let us
denote by $\alpha_{1} = 4(Deg_{y_{1}}(LT(f)) + \ldots +
Deg_{y_{n_{1}}}(LT(f)))$, $\alpha_{2} = \alpha_{1} + 2(Deg_{y_{n_{1}
+ 1}}(LT(f)) + \ldots + Deg_{y_{n_{1} + n_{2}}}(LT(f))) + 2(n_{2} -
n_{1})$, $\alpha_{3} = \alpha_{2} + 2.n_{3}$, $\alpha_{4} =
\alpha_{3} + 2.(Deg_{z_{1}}(LT(f)) + \ldots +
Deg_{z_{m_{1}}}(LT(f)))$, $\alpha_{5} = 2.(Deg_{z_{1}}(LT(f)) +
\ldots + Deg_{z_{m_{1}}}(LT(f))) - 1$, $\alpha_{6} = \alpha_{5} + 2(
Deg_{z_{m_{1} + 1}}(LT(f)) + \ldots + Deg_{z_{m_{1} +
m_{2}}}(LT(f)))$, $\alpha_{7} = \alpha_{6} + 2.m_{3}$ e $\alpha_{8}
= \alpha_{3} + \alpha_{5} + 2.(Deg_{z_{m_{1} + 1}}(LT(f)) + \ldots +
Deg_{z_{m_{1} + m_{2}}}(LT(f))) + 2(2 - (m_{2} + 1)) - 1$

Let $n_1+n_2+n_3\leq l$ and $m_1+m_2+m_3\leq m$.

We consider the graded homomorphism $\phi: F\langle
y_{1},\ldots,y_{n_{1} + n_{2} + n_{3}}, z_{1}, \ldots, z_{m_{1} +
m_{2} + m_{3}} \rangle \rightarrow E_\infty$:

\begin{center}
$y_{1} \mapsto e_{2}e_{4} + \ldots +
e_{4.Deg_{y_{1}}(LT(f)) - 2}e_{4.Deg_{y_{1}}(LT(f))}$ \\

$\ldots$ \\

$y_{n_{1}} \mapsto e_{4(Deg_{y_{1}}(LT(f)) + \ldots + Deg_{y_{n_{1}
- 1}}(LT(f))) + 2}e_{4(Deg_{y_{1}}(LT(f)) + \ldots + Deg_{y_{n_{1} -
1}}(LT(f))) + 4} + \ldots + e_{4(Deg_{y_{1}}(LT(f)) + \ldots +
Deg_{y_{n_{1}}}(LT(f))) - 2}e_{4(Deg_{y_{1}}(LT(f)) + \ldots +
Deg_{y_{n_{1}}}(LT(f)))}$ \\

$y_{n_{1} + 1} \mapsto e_{\alpha_{1} + 2} + e_{\alpha_{1} +
4}e_{\alpha_{1} + 6} + \ldots + e_{\alpha_{1} + 4. Deg_{y_{n_{1} +
1}}(LT(f))}e_{\alpha_{1} +  4.Deg_{y_{n_{1} + 1}}(LT(f)) + 2}$ \\

$\ldots$ \\

$y_{n_{1} + n_{2}} \mapsto e_{\alpha_{1} + 4(Deg_{y_{n_{1} +
1}}(LT(f)) + \ldots + Deg_{y_{n_{1} + n_{2} - 1}}(LT(f))) + 2(n_{2}
- n_{1})} + e_{\alpha_{1} + 4(Deg_{y_{n_{1} + 1}}(LT(f)) + \ldots +
Deg_{y_{n_{1} + n_{2} - 1}}(LT(f))) + 2(n_{2} - n_{1}) +
2}e_{\alpha_{1} + 4(Deg_{y_{n_{1} + 1}}(LT(f)) + \ldots +
Deg_{y_{n_{1} + n_{2} - 1}}(LT(f))) + 2(n_{2} - n_{1}) + 4} + \ldots
+ e_{\alpha_{1} + 4(Deg_{y_{n_{1} + 1}}(LT(f)) + \ldots +
Deg_{y_{n_{1} + n_{2}}}(LT(f))) + 2(n_{2} - n_{1} - 1)}e_{\alpha_{1}
+ 4(Deg_{y_{n_{1} + 1}}(LT(f)) + \ldots
+ Deg_{y_{n_{2} + n_{1}}}(LT(f))) + 2(n_{2} - n_{1})}$ \\

$y_{n_{1} + n_{2} + 1} \mapsto e_{\alpha_{2} + 2}$ \\

$\ldots$ \\

$y_{n_{1} + n_{2} + n_{3}} \mapsto e_{\alpha_{2} + 2.n_{3}}$ \\

$z_{1} \mapsto e_{1}e_{\alpha_{3} + 2} + \ldots +
e_{2.(Deg_{z_{1}}(LT(f))) - 1}e_{\alpha_{3} + 2(Deg_{z_{1}}(LT(f)))}$\\

$\ldots$ \\

$z_{m_{1}} \mapsto e_{2.(Deg_{z_{1}}(LT(f)) + \ldots + Deg_{z_{m_{1}
- 1}}(LT(f))) + 1}e_{\alpha_{3} + 2.(Deg_{z_{1}}(LT(f)) + \ldots +
Deg_{z_{m_{1} - 1}}(LT(f))) + 2} + \ldots + e_{2.(Deg_{z_{1}}(LT(f))
+ \ldots + Deg_{z_{m_{1}}}(LT(f))) - 1}e_{\alpha_{3} +
2.(Deg_{z_{1}}(LT(f)) + \ldots +
Deg_{z_{m_{1}}}(LT(f)))}$ \\

$z_{m_{1} + 1} \mapsto e_{\alpha_{5} + 2} + e_{\alpha_{5} +
4}e_{\alpha_{3} + \alpha_{5} + 3} +
\ldots + e_{\alpha_{5} + 2Deg_{z_{m_{1} + 1}}(LT(f))}e_{\alpha_{3} + \alpha_{5} + 2Deg_{z_{m_{1} + 1}}(LT(f)) - 1}$ \\

$\ldots$ \\

$z_{m_{1} + m_{2}} \mapsto e_{\alpha_{5} +
2(Deg_{z_{m_{1}+1}}(LT(f)) + \ldots + Deg_{z_{m_{1} + m_{2} -
1}}(LT(f))) + 2} + e_{\alpha_{5} + 2( Deg_{z_{m_{1} + 1}}(LT(f)) +
\ldots + Deg_{z_{m_{1} + m_{2} - 1}}(LT(f))) + 4}e_{\alpha_{3} +
\alpha_{5} + 2.(Deg_{z_{m_{1} + 1}}(LT(f)) + \ldots + Deg_{z_{m_{1}
+ m_{2} - 1}}(LT(f))) + 2(2 - m_{2}) + 1} + \ldots + e_{\alpha_{5} +
2( Deg_{z_{m_{1} + 1}}(LT(f)) + \ldots + Deg_{z_{m_{1} +
m_{2}}}(LT(f)))}e_{\alpha_{3} + \alpha_{5} + 2.(Deg_{z_{m_{1} +
1}}(LT(f)) + \ldots + Deg_{z_{m_{1} +
m_{2}}}(LT(f))) + 2(2 - (m_{2} + 1)) - 1}$ \\

$z_{m_{1} + m_{2} + 1} \mapsto e_{\alpha_{6} + 2}$ \\

$\ldots$ \\

$z_{m_{1} + m_{2} + m_{3}} \mapsto e_{\alpha_{6} + 2.m_{3}}$
\end{center}
Notice that $\phi(LT(f)) =
\lambda\lambda_{j}e_{1}.\cdots.e_{\alpha_{7}}e_{2}.\cdots.e_{\alpha_{8}}$
for some $\lambda \in F - \{0\}$.

Suppose there exists $M_{i} \in \{M_{1},\ldots,M_{n}\}$ such that
$M_{i} < LT(f)$ e $\lambda_{i} \neq 0$. We claim that none of the
support of any summand of $dom(\phi(M_{i}))$ contains $supp(g)$,
where $g = e_{1}.\cdots.e_{\alpha_{7}}e_{2}.\cdots.e_{\alpha_{8}}$.

\begin{description}
\item Case 1: $deg (M_{i}) < deg (LT(f))$.
\subitem In this case there exists a variable $x \in X$ such that
$Deg_{x}(LT(f)) > Deg_{x}(M_{i})$. So by the definition of $\phi$, we have $supp(\phi(x))$ is not contained in any of the summand of $(dom(\phi(M_{i})))$.
\item Case 2: $deg(M_{i}) = deg(LT(f))$, but $beg(M_{i}) <_{lex-rig}
beg(LT(f))$. \subitem In this case there exists a variable $x \in
X$ such that $Deg_{x}(beg(LT(f))) > Deg_{x}(beg(M_{i}))$. Again by the definition of $\phi$, we have $supp(\phi(x))$ is not contained in the support of any of the summand of $\phi(M_{i})$.
\item Case 3: $deg(M_{i}) = deg(LT(f)), beg(M_{i}) = beg(LT(f))$, but $\psi(M_{i}) <_{lex-rig} \psi(LT(f))$.
In this case there exists a variable $x \in X$ such that
$Deg_{x}(\psi(LT(f))) = 1$ and $Deg_{x}(\psi(M_{i}))= 0$.
Analogously to Case 1, $supp(\phi(x))$ is not contained in the
support of any summand of $\phi(M_{i})$.
\end{description}

In light of the previous cases we have $dom(\phi(f)) =
\phi(LT(f))$. On the other hand $\phi(f) = 0$, then $dom(\phi(f)) = 0$. This means $\lambda.\lambda_{j}g =
0$ which is a contradiction and the proof is complete.
\end{proof}

Combining Lemma \ref{combss} and Proposition \ref{id3} and in light of Lemma \ref{combssstar2} we have the following result.

\begin{teo}
For each $n\in\N$ and $n_1,n_2\in\N$ such that $n_1+n_2=n$ we have \[c_{(n_1,n_2)}(SS)\leq c_{(n_1,n_2)}(E_{\infty},W(n))\leq c_{(n_1,n_2)}^\circ(SS).\]
\end{teo}

\subsection{$E_{k^{*}}$}

\

\begin{lema}\label{segundo lema}
The polynomials $SS1 \cap F\langle
y_{1},\ldots,y_{l},z_{1},\ldots,z_{m}\rangle$ form a linearly independent set modulo $T_{2}(E_{k^{*}})$.
\end{lema}
\begin{proof}
The proof is analogous to that of Lemma \ref{primeiro lema}.
We shall define $\alpha = k + 2(Deg_{y_{1}}(LT(f)) + \ldots +
Deg_{y_{n_{1}}}(LT(f))), \alpha_{1} = \alpha + 2(Deg_{y_{n_{1} +
1}}(LT(f)) + \ldots + Deg_{y_{n_{1} + n_{2}}}(LT(f))), \alpha_{2} =
\alpha_{1} + n_{3}, \alpha_{3} =  Deg_{z_{1}}(LT(f)) + \ldots +
Deg_{z_{m_{1}}}(LT(f))$.

Let $n_1+n_2+n_3\leq l$ and $m_1+m_2+m_3\leq m$. Then we consider the next graded homomorphism:

\begin{center}
$\phi: F\langle y_{1},\ldots,y_{n_{1}+n_{2}+n_{3}},
z_{1},\ldots,z_{m_{1}+m_{2}+m_{3}} \rangle \rightarrow E$ \\
$y_{1} \mapsto e_{k+1}e_{k+2} + \ldots + e_{k +
2(Deg_{y_{1}}(LT(f))) - 1}e_{k + 2(Deg_{y_{1}}(LT(f)))}$\\
$\ldots$ \\
$y_{n_{1}} \mapsto e_{k + 2(Deg_{y_{1}}(LT(f)) + \ldots +
Deg_{y_{n_{1} - 1}}(LT(f))) + 1}e_{k + 2(Deg_{y_{1}}(LT(f)) + \ldots
+ Deg_{y_{n_{1} - 1}}(LT(f))) + 2} + \ldots + e_{k +
2(Deg_{y_{1}}(LT(f)) + \ldots + Deg_{y_{n_{1}}}(LT(f))) - 1}e_{k +
2(Deg_{y_{1}}(LT(f)) + \ldots +
Deg_{y_{n_{1}}}(LT(f)))}$ \\
$y_{n_{1} + 1} \mapsto e_{\alpha + 1} + e_{\alpha + 2}e_{\alpha + 3}
+ \ldots + e_{\alpha + 2Deg_{y_{n_{1} + 1}}(LT(f))}e_{\alpha + 2Deg_{y_{n_{1} + 1}}(LT(f)) + 1}$ \\
$\ldots$ \\
$y_{n_{1} + n_{2}} \mapsto e_{\alpha + 2(Deg_{y_{n_{1} + 1}}(LT(f))
+ \ldots + Deg_{y_{n_{1} + n_{2} - 1}}(LT(f))) + n_{2}} + e_{\alpha
+ 2(Deg_{y_{n_{1}+1}}(LT(f)) + \ldots + Deg_{y_{n_{1} + n_{2} -
1}}(LT(f))) + n_{2} + 1}e_{\alpha + 2(Deg_{y_{n_{1}+1}}(LT(f)) +
\ldots + Deg_{y_{n_{1} + n_{2} - 1}}(LT(f))) + n_{2} + 2} + \ldots +
e_{\alpha + 2(Deg_{y_{n_{1} + 1}}(LT(f)) + \ldots + Deg_{y_{n_{1} +
n_{2}}}(LT(f))) + n_{2} - 1}e_{\alpha + 2(Deg_{y_{n_{1} + 1}}(LT(f))
+ \ldots
+ Deg_{y_{n_{1} + n_{2}}}(LT(f))) + n_{2}} $\\
$\ldots$\\
$y_{n_{1} + n_{2} + n_{3}} \mapsto e_{\alpha_{1} + n_{3}}$\\
$z_{1} \mapsto e_{1}e_{\alpha_{2}} + \ldots +
e_{Deg_{z_{1}}(LT(f))}e_{\alpha_{2} + Deg_{z_{1}}(LT(f))}$\\
$\ldots$ \\
$z_{m_{1}} \mapsto e_{Deg_{z_{1}}(LT(f)) + \ldots + Deg_{z_{m_{1} -
1}}(LT(f)) + 1 }e_{\alpha_{2} + Deg_{z_{1}}(LT(f)) + \ldots +
Deg_{z_{m_{1} - 1}}(LT(f)) + 1} + \ldots + e_{Deg_{z_{1}}(LT(f)) +
\ldots + Deg_{z_{m_{1}}}(LT(f))}e_{\alpha_{2} + Deg_{z_{1}}(LT(f)) +
\ldots + Deg_{z_{m_{1}}}(LT(f))}$\\
$z_{m_{1} + 1} \mapsto e_{\alpha_{3} + 1} + e_{\alpha_{3} +
2}e_{\alpha_{2} + \alpha_{3} + 1} + \ldots + e_{\alpha_{3} +
Deg_{z_{m_{1} + 1}}(LT(f))}e_{\alpha_{2} + \alpha_{3} +
Deg_{z_{m_{1}
+ 1}}(LT(f)) - 1}$ \\
$\ldots$ \\
$z_{m_{1} + m_{2}} \mapsto e_{\alpha_{3} + Deg_{z_{m_{1} +
1}}(LT(f)) + \ldots + Deg_{z_{m_{1} + m_{2} - 1}}(LT(f)) + 1} +
e_{\alpha_{3} + Deg_{z_{m_{1} + 1}}(LT(f)) + \ldots + Deg_{z_{m_{1}
+ m_{2} - 1}}(LT(f)) + 2}e_{\alpha_{2} + \alpha_{3} + Deg_{z_{m_{1}
+ 1}}(LT(f)) + \ldots + Deg_{z_{m_{1} + m_{2} - 1}}(LT(f)) - m_{2} +
2} + \ldots + e_{\alpha_{3} + Deg_{z_{m_{1} + 1}}(LT(f)) + \ldots +
Deg_{z_{m_{1} + m_{2}}}(LT(f))}e_{\alpha_{2} + \alpha_{3} +
Deg_{z_{m_{1} + 1}}(LT(f))
+ \ldots + Deg_{z_{m_{1} + m_{2}}}(LT(f)) - m_{2}}$ \\
$z_{m_{1} + m_{2} + 1} \mapsto e_{\alpha_{3} + Deg_{z_{m_{1} +
1}}(LT(f)) + \ldots + Deg_{z_{m_{1} + m_{2}}}(LT(f)) + 1}$ \\
$\ldots$ \\
$z_{m_{1} + m_{2} + m_{3}} \mapsto e_{\alpha_{3} + Deg_{z_{m_{1} +
1}}(LT(f)) + \ldots + Deg_{z_{m_{1} + m_{2}}}(LT(f)) + m_{3}}$.
\end{center}
\end{proof}

Combining Lemma \ref{combss1} and Proposition \ref{proposicao 9.2} and in light of Lemma \ref{combssstar2} we have the following result.

\begin{teo}
For each $n\in\N$ and $n_1,n_2\in\N$ such that $n_1+n_2=n$ we have \[c_{(n_1,n_2)}(SS1)\leq c_{(n_1,n_2)}(E_{k^*},W(n))\leq c_{(n_1,n_2)}^\circ(SS1).\]
\end{teo}

\subsection{$E_{k}, \ \ k \geq 1$}

\

We have the following.

\begin{lema}\label{terceiro lema}
The polynomials $SS3 \cap F\langle
y_{1},\ldots,y_{l},z_{1},\ldots,z_{m}\rangle$ form a linearly independent set modulo $T_{2}(E_{k})$.
\end{lema}
\begin{proof}
Let $f = \sum_{i = 1}^{n}\lambda_{i}u_{i} \equiv 0$ be a linear combination modulo $T_{2}(E_{k})$ of elements of $SS3 \cap F\langle
y_{1},\ldots,y_{l},z_{1},\ldots,z_{m}\rangle$. As usual we may suppose each variable of $f$ appearing in each summand with degree at least 1. Let us suppose by contradiction that there exists $M_{j} \in \{M_{1},\ldots,M_{n}\}$, such that $\lambda_{j} \neq
0$. Moreover let us assume, without loss of generality, $M_{j} = LT(f)$. We also denote \[\alpha_{1} = k + 2(Deg_{y_{1}}(LT(f)) +
\ldots + Deg_{y_{n_{1}}}(LT(f))),\] \[\alpha_{2} = \alpha_{1} +
2(Deg_{y_{n_{1} + 1}}(LT(f)) + \ldots + Deg_{y_{n_{1} + n_{2} -
1}}(LT(f))) - 2n_{2},\] \[\alpha_{3} = \alpha_{2} + Deg_{z_{1}}(LT(f)) +
\ldots + Deg_{z_{m_{1}}}(LT(f)),\] \[\alpha_{4} = n_{2} + n_{3} +
Deg_{z_{1}}(LT(f)) + \ldots + Deg_{z_{m_{1}}}(LT(f)),\] \[\alpha_{5} =
\alpha_{3} + Deg_{z_{m_{1}+1}}(LT(f)) + \ldots + Deg_{z_{m_{1} +
m_{2}}}(LT(f)) + m_{3},\] \[\alpha_{6} = \alpha_{4} +
Deg_{z_{m_{1}+1}}(LT(f)) + \ldots + Deg_{z_{m_{1} + m_{2}}}(LT(f)) -
m_{2}.\]

We have to study three cases.

\begin{description}
\item Case 1: $deg_{Z}(beg(LT(f))) + deg_{Y}(\psi(LT(f))) \leq k$;
\item Case 2: $deg_{Z}(beg(LT(f))) + deg_{Y}(\psi(LT(f))) = k + 1$ and none of the elements of $\{u_{1},\ldots,u_{n}\}$ is bad;
\item Case 3: $deg_{Z}(beg(LT(f))) + deg_{Y}(\psi(LT(f))) = k + 1$ and there exists $u_{i} \in \{u_{1},\ldots,u_{n}\}$ such that $\lambda_{i} \neq
0$ and $u_{i}$ is bad.
\end{description}

In all of the three cases we set $n_1+n_2+n_3\leq l$ $m_1+m_2+m_3\leq m$.

Let us analyze Case 1. So we consider the following graded
homomorphism:

\begin{center}
$\phi: F\langle y_{1},\ldots,y_{n_{1}+n_{2}+n_{3}},z_{1},\ldots,z_{m_{1}+m_{2}+m_{3}}\rangle \rightarrow E_k$\\

$y_{1} \mapsto e_{k+1}e_{k+2} + \ldots + e_{k + 2Deg_{y_{1}}(LT(f))
- 1}e_{k + 2Deg_{y_{1}}(LT(f))}$ \\
$\ldots$ \\
$y_{n_{1}} \mapsto e_{k + 2(Deg_{y_{1}}(LT(f)) + \ldots +
Deg_{y_{n_{1} - 1}}(LT(f))) + 1}e_{k + 2(Deg_{y_{1}}(LT(f)) + \ldots
+ Deg_{y_{n_{1} - 1}}(LT(f))) + 2} + \ldots + e_{k +
2(Deg_{y_{1}}(LT(f)) + \ldots + Deg_{y_{n_{1}}}(LT(f))) - 1}e_{k +
2(Deg_{y_{1}}(LT(f)) + \ldots + Deg_{y_{n_{1}}}(LT(f)))}$ \\
$y_{n_{1} + 1} \mapsto e_{1} + e_{\alpha_{1} + 1}e_{\alpha_{1} + 2}
+ \ldots + e_{\alpha_{1} + 2Deg_{y_{n_{1} + 1}}(LT(f)) -
3}e_{\alpha_{1} +
2Deg_{y_{n_{1} + 1}}(LT(f)) - 2}$ \\
$\ldots$ \\
$y_{n_{1} + n_{2}} \mapsto e_{n_{2}} + e_{\alpha_{1} +
2(Deg_{y_{n_{1} + 1}}(LT(f)) + \ldots + Deg_{y_{n_{1} + n_{2} -
1}}(LT(f))) + 2(1 - n_{2}) + 1}e_{\alpha_{1} + 2(Deg_{y_{n_{1} +
1}}(LT(f)) + \ldots + Deg_{y_{n_{1} + n_{2} - 1}}(LT(f))) + 2(1 -
n_{2}) + 2} + \ldots + e_{\alpha_{1} + 2(Deg_{y_{n_{1} + 1}}(LT(f))
+ \ldots + Deg_{y_{n_{1} + n_{2}}}(LT(f))) - 2n_{2} -
1}e_{\alpha_{1} + 2(Deg_{y_{n_{1} + 1}}(LT(f)) + \ldots +
Deg_{y_{n_{1} + n_{2} - 1}}(LT(f))) - 2n_{2}}$ \\
$y_{n_{1} + n_{2} + 1} \mapsto e_{n_{2} + 1}$ \\
$\ldots$ \\
$y_{n_{1} + n_{2} + n_{3}} \mapsto e_{n_{2} + n_{3}}$\\
$z_{1} \mapsto e_{\alpha_{2} + 1}e_{n_{2} + n_{3} + 1} + \ldots +
e_{\alpha_{2} + Deg_{z_{1}}(LT(f))}e_{n_{2} + n_{3} +
Deg_{z_{1}}(LT(f))}$ \\
$\ldots$ \\
$z_{m_{1}} \mapsto e_{\alpha_{2} + Deg_{z_{1}}(LT(f)) + \ldots +
Deg_{z_{m_{1} - 1}}(LT(f)) + 1}e_{n_{2} + n_{3} + Deg_{z_{1}}(LT(f))
+ \ldots + Deg_{z_{m_{1} - 1}}(LT(f)) + 1} + \ldots + e_{\alpha_{2}
+ Deg_{z_{1}}(LT(f)) + \ldots + Deg_{z_{m_{1}}}(LT(f))}e_{n_{2} +
n_{3} + Deg_{z_{1}}(LT(f)) + \ldots + Deg_{z_{m_{1}}}(LT(f))}$ \\
$\ldots$\\
$z_{m_{1} + 1} \mapsto e_{\alpha_{3}} + e_{\alpha_{3} + 1}e_{\alpha_{4} + 1} + \ldots + e_{\alpha_{3} + Deg_{z_{m_{1} + 1}}(LT(f))}e_{\alpha_{4} + Deg_{z_{m_{1} + 1}}(LT(f)) - 1}$ \\
$\ldots$ \\
$z_{m_{1} + m_{2}} \mapsto e_{\alpha_{3} + Deg_{z_{m_{1}+1}}(LT(f))
+ \ldots + Deg_{z_{m_{1} + m_{2} - 1}}(LT(f))} + e_{\alpha_{4} +
Deg_{z_{m_{1}+1}}(LT(f)) + \ldots + Deg_{z_{m_{1} + m_{2} -
1}}(LT(f)) - m_{2} + 2}e_{\alpha_{3} + Deg_{z_{m_{1}+1}}(LT(f)) +
\ldots + Deg_{z_{m_{1} + m_{2} - 1}}(LT(f)) + 1} + \ldots +
e_{\alpha_{4} + Deg_{z_{m_{1}+1}}(LT(f)) + \ldots + Deg_{z_{m_{1} + m_{2}}}(LT(f)) - m_{2}}e_{\alpha_{3} + Deg_{z_{m_{1}+1}}(LT(f)) + \ldots + Deg_{z_{m_{1} + m_{2}}}(LT(f))}$ \\
$z_{m_{1} + m_{2} + 1} \mapsto e_{\alpha_{3} + Deg_{z_{m_{1}+1}}(LT(f)) + \ldots + Deg_{z_{m_{1} + m_{2}}}(LT(f)) + 1}$ \\
$z_{m_{1} + m_{2} + m_{3}} \mapsto e_{\alpha_{3} +
Deg_{z_{m_{1}+1}}(LT(f)) + \ldots + Deg_{z_{m_{1} + m_{2}}}(LT(f)) +
m_{3}}$
\end{center}
We have $\phi(LT(f)) = \lambda.\lambda_{j}
e_{1}.\cdots.e_{\alpha_{6}}.e_{k+1}.\cdots.e_{\alpha_{5}}$ for some $\lambda \in F - \{0\}$. Suppose there exists $M_{i} <
LT(f)$, such that $\lambda_{i} \neq 0$, then there exists a variable $x
\in \{y_{1},\ldots,z_{m_{1}+m_{2}+m_{3}}\}$ such that $supp(\phi(x))$
is not contained in the support of any of the summands of $dom(\phi(M_{i}))$. We obtain $dom(\phi(f)) =
\lambda.\lambda_{i}e_{1}.\cdots.e_{\alpha_{6}}.e_{k+1}.\cdots.e_{\alpha_{5}}
= 0$ which is a contradiction.

Now we study Case 2. Let us consider the following graded homomorphism:

\begin{center}
$\phi: F\langle y_{1},\ldots,y_{n_{1}+n_{2}+n_{3}},z_{1},\ldots,z_{m_{1}+m_{2}+m_{3}}\rangle \rightarrow E_k$\\

$y_{1} \mapsto e_{k+1}e_{k+2} + \ldots + e_{k + 2Deg_{y_{1}}(LT(f))
- 1}e_{k + 2Deg_{y_{1}}(LT(f))}$ \\
$\ldots$ \\
$y_{n_{1}} \mapsto e_{k + 2(Deg_{y_{1}}(LT(f)) + \ldots +
Deg_{y_{n_{1} - 1}}(LT(f))) + 1}e_{k + 2(Deg_{y_{1}}(LT(f)) + \ldots
+ Deg_{y_{n_{1} - 1}}(LT(f))) + 2} + \ldots + e_{k +
2(Deg_{y_{1}}(LT(f)) + \ldots + Deg_{y_{n_{1}}}(LT(f))) - 1}e_{k +
2(Deg_{y_{1}}(LT(f)) + \ldots + Deg_{y_{n_{1}}}(LT(f)))}$ \\
$y_{n_{1} + 1} \mapsto e_{1} + e_{\alpha_{1} + 1}e_{\alpha_{1} + 2}
+ \ldots + e_{\alpha_{1} + 2Deg_{y_{n_{1} + 1}}(LT(f)) -
3}e_{\alpha_{1} +
2Deg_{y_{n_{1} + 1}}(LT(f)) - 2}$ \\
$\ldots$ \\
$y_{n_{1} + n_{2}} \mapsto e_{n_{2}} + e_{\alpha_{1} +
2(Deg_{y_{n_{1} + 1}}(LT(f)) + \ldots + Deg_{y_{n_{1} + n_{2} -
1}}(LT(f))) + 2(1 - n_{2}) + 1}e_{\alpha_{1} + 2(Deg_{y_{n_{1} +
1}}(LT(f)) + \ldots + Deg_{y_{n_{1} + n_{2} - 1}}(LT(f))) + 2(1 -
n_{2}) + 2} + \ldots + e_{\alpha_{1} + 2(Deg_{y_{n_{1} + 1}}(LT(f))
+ \ldots + Deg_{y_{n_{1} + n_{2}}}(LT(f))) - 2n_{2} -
1}e_{\alpha_{1} + 2(Deg_{y_{n_{1} + 1}}(LT(f)) + \ldots +
Deg_{y_{n_{1} + n_{2} - 1}}(LT(f))) - 2n_{2}}$ \\
$y_{n_{1} + n_{2} + 1} \mapsto e_{n_{2} + 1}$ \\
$\ldots$ \\
$y_{n_{1} + n_{2} + n_{3}} \mapsto e_{n_{2} + n_{3}}$\\
$z_{1} \mapsto e_{\alpha_{2} + 1} + \ldots + e_{\alpha_{2} +
Deg_{z_{1}}(LT(f))}e_{n_{2} + n_{3} +
Deg_{z_{1}}(LT(f)) - 1}$ \\
$\ldots$ \\
$z_{m_{1}} \mapsto e_{\alpha_{2} + Deg_{z_{1}}(LT(f)) + \ldots +
Deg_{z_{m_{1} - 1}}(LT(f))}e_{n_{2} + n_{3} + Deg_{z_{1}}(LT(f)) +
\ldots + Deg_{z_{m_{1} - 1}}(LT(f))} + \ldots + e_{\alpha_{2} +
Deg_{z_{1}}(LT(f)) + \ldots + Deg_{z_{m_{1}}}(LT(f))}e_{n_{2} +
n_{3} + Deg_{z_{1}}(LT(f)) + \ldots + Deg_{z_{m_{1}}}(LT(f)) - 1}$ \\
$\ldots$\\
$z_{m_{1} + 1} \mapsto e_{\alpha_{3}} + e_{\alpha_{3} + 1}e_{\alpha_{4}} + \ldots + e_{\alpha_{3} + Deg_{z_{m_{1} + 1}}(LT(f))}e_{\alpha_{4} + Deg_{z_{m_{1} + 1}}(LT(f)) - 2}$ \\
$\ldots$ \\
$z_{m_{1} + m_{2}} \mapsto e_{\alpha_{3} + Deg_{z_{m_{1}+1}}(LT(f))
+ \ldots + Deg_{z_{m_{1} + m_{2} - 1}}(LT(f))} + e_{\alpha_{4} +
Deg_{z_{m_{1}+1}}(LT(f)) + \ldots + Deg_{z_{m_{1} + m_{2} -
1}}(LT(f)) - m_{2} + 1}e_{\alpha_{3} + Deg_{z_{m_{1}+1}}(LT(f)) +
\ldots + Deg_{z_{m_{1} + m_{2} - 1}}(LT(f)) + 1} + \ldots +
e_{\alpha_{4} + Deg_{z_{m_{1}+1}}(LT(f)) + \ldots + Deg_{z_{m_{1} + m_{2}}}(LT(f)) - m_{2} - 1}e_{\alpha_{3} + Deg_{z_{m_{1}+1}}(LT(f)) + \ldots + Deg_{z_{m_{1} + m_{2}}}(LT(f))}$ \\
$z_{m_{1} + m_{2} + 1} \mapsto e_{\alpha_{3} + Deg_{z_{m_{1}+1}}(LT(f)) + \ldots + Deg_{z_{m_{1} + m_{2}}}(LT(f)) + 1}$ \\
$z_{m_{1} + m_{2} + m_{3}} \mapsto e_{\alpha_{3} +
Deg_{z_{m_{1}+1}}(LT(f)) + \ldots + Deg_{z_{m_{1} + m_{2}}}(LT(f)) +
m_{3}}$
\end{center}
Notice that $\phi(LT(f)) = \lambda.\lambda_{j}
e_{1}.\cdots.e_{\alpha_{6} - 1}.e_{k+1}.\cdots.e_{\alpha_{5}}$ for
some $\lambda \in F - \{0\}$. Let us suppose that there exists
$M_{i} < LT(F)$ such that $\lambda_{i} \neq 0$. If for some $x\in X$
we have $Deg_{x} M_{i} < Deg_{x} LT(f)$, then it is easy to see that
non of the supports of any summand of $dom(\phi(M_{i}))$ contains
$supp(\phi(x))$.

From now on we shall consider the case in which $Deg_{x} M_{i} =
Deg_{x} LT(f)$ for every $x \in X$. In this case if $beg LT(f) = beg
u_{i}$, we have $LT(f) = u_{i}$. Hence we have only one case to be
studied, i.e., $beg(M_{i}) <_{lex-rig} beg(LT(f))$. Due to the fact
that $M_{i}$ is not bad, there exists a variable $x \in X -
pr(z)(LT(F))$ such that $Deg_{x}(beg(M_{i})) < Deg_{x}(beg(LT(f)))$.
Then we have none of the supports of any summand of $dom
(\phi(M_{i}))$ contains $supp(\phi(x))$.

In order to complete the proof for Case 2 it is enough to repeat verbatim the proof of Case 1. We shall obtain  $dom(\phi(f)) =
\phi(LT(f)) = 0$ which is a contradiction.

Finally we consider Case 3. Let us denote by $\lambda_{k} \neq 0$
the coefficient associated to $LBT(f)$. Let us consider the
following graded homomorphism:

\begin{center}
$\phi: F\langle y_{1},\ldots,y_{n_{1}+n_{2}+n_{3}},z_{1},\ldots,z_{m_{1}+m_{2}+m_{3}}\rangle \rightarrow E$\\

$y_{1} \mapsto e_{k+1}e_{k+2} + \ldots + e_{k + 2Deg_{y_{1}}(LBT(f))
- 1}e_{k + 2Deg_{y_{1}}(LBT(f))}$ \\
$\ldots$ \\
$y_{n_{1}} \mapsto e_{k + 2(Deg_{y_{1}}(LBT(f)) + \ldots +
Deg_{y_{n_{1} - 1}}(LBT(f))) + 1}e_{k + 2(Deg_{y_{1}}(LBT(f)) +
\ldots + Deg_{y_{n_{1} - 1}}(LBT(f))) + 2} + \ldots + e_{k +
2(Deg_{y_{1}}(LBT(f)) + \ldots + Deg_{y_{n_{1}}}(LBT(f))) - 1}e_{k +
2(Deg_{y_{1}}(LBT(f)) + \ldots + Deg_{y_{n_{1}}}(LBT(f)))}$ \\
$y_{n_{1} + 1} \mapsto e_{1} + e_{\alpha_{1} + 1}e_{\alpha_{1} + 2}
+ \ldots + e_{\alpha_{1} + 2Deg_{y_{n_{1} + 1}}(LBT(f)) -
3}e_{\alpha_{1} +
2Deg_{y_{n_{1} + 1}}(LBT(f)) - 2}$ \\
$\ldots$ \\
$y_{n_{1} + n_{2}} \mapsto e_{n_{2}} + e_{\alpha_{1} +
2(Deg_{y_{n_{1} + 1}}(LBT(f)) + \ldots + Deg_{y_{n_{1} + n_{2} -
1}}(LBT(f))) + 2(1 - n_{2}) + 1}e_{\alpha_{1} + 2(Deg_{y_{n_{1} +
1}}(LBT(f)) + \ldots + Deg_{y_{n_{1} + n_{2} - 1}}(LBT(f))) + 2(1 -
n_{2}) + 2} + \ldots + e_{\alpha_{1} + 2(Deg_{y_{n_{1} + 1}}(LBT(f))
+ \ldots + Deg_{y_{n_{1} + n_{2}}}(LBT(f))) - 2n_{2} -
1}e_{\alpha_{1} + 2(Deg_{y_{n_{1} + 1}}(LBT(f)) + \ldots +
Deg_{y_{n_{1} + n_{2} - 1}}(LBT(f))) - 2n_{2}}$ \\
$y_{n_{1} + n_{2} + 1} \mapsto e_{n_{2} + 1}$ \\
$\ldots$ \\
$y_{n_{1} + n_{2} + n_{3}} \mapsto e_{n_{2} + n_{3}}$\\
$z_{2} \mapsto e_{\alpha_{2} + 1}e_{n_{2} + n_{3} + 1} + \ldots +
e_{\alpha_{2} + Deg_{z_{1}}(LBT(f))}e_{n_{2} + n_{3} +
Deg_{z_{1}}(LBT(f))}$ \\
$\ldots$ \\
$z_{m_{1}} \mapsto e_{\alpha_{2} + Deg_{z_{1}}(LBT(f)) + \ldots +
Deg_{z_{m_{1} - 1}}(LBT(f)) + 1}e_{n_{2} + n_{3} +
Deg_{z_{1}}(LBT(f)) + \ldots + Deg_{z_{m_{1} - 1}}(LBT(f)) + 1} +
\ldots + e_{\alpha_{2} + Deg_{z_{1}}(LBT(f)) + \ldots +
Deg_{z_{m_{1}}}(LBT(f))}e_{n_{2} +
n_{3} + Deg_{z_{1}}(LBT(f)) + \ldots + Deg_{z_{m_{1}}}(LBT(f))}$ \\
$\ldots$\\
$z_{1} \mapsto e_{\alpha_{3}} + e_{\alpha_{3} + 1}e_{\alpha_{4} + 1} + \ldots + e_{\alpha_{3} + Deg_{z_{m_{1} + 1}}(LBT(f))}e_{\alpha_{4} + Deg_{z_{m_{1} + 1}}(LBT(f)) - 1}$ \\
$\ldots$ \\
$z_{m_{1} + m_{2}} \mapsto e_{\alpha_{3} + Deg_{z_{m_{1}+1}}(LBT(f))
+ \ldots + Deg_{z_{m_{1} + m_{2} - 1}}(LBT(f))} + e_{\alpha_{4} +
Deg_{z_{m_{1}+1}}(LBT(f)) + \ldots + Deg_{z_{m_{1} + m_{2} -
1}}(LBT(f)) - m_{2} + 2}e_{\alpha_{3} + Deg_{z_{m_{1}+1}}(LBT(f)) +
\ldots + Deg_{z_{m_{1} + m_{2} - 1}}(LBT(f)) + 1} + \ldots +
e_{\alpha_{4} + Deg_{z_{m_{1}+1}}(LBT(f)) + \ldots + Deg_{z_{m_{1} + m_{2}}}(LBT(f)) - m_{2}}e_{\alpha_{3} + Deg_{z_{m_{1}+1}}(LBT(f)) + \ldots + Deg_{z_{m_{1} + m_{2}}}(LBT(f))}$ \\
$z_{m_{1} + m_{2} + 1} \mapsto e_{\alpha_{3} + Deg_{z_{m_{1}+1}}(LBT(f)) + \ldots + Deg_{z_{m_{1} + m_{2}}}(LBT(f)) + 1}$ \\
$z_{m_{1} + m_{2} + m_{3}} \mapsto e_{\alpha_{3} +
Deg_{z_{m_{1}+1}}(LBT(f)) + \ldots + Deg_{z_{m_{1} + m_{2}}}(LBT(f))
+ m_{3}}$
\end{center}
Notice that $\phi(LBT(f)) = \lambda.\lambda_{k}
e_{1}.\ldots.e_{\alpha_{6}}e_{k+1}\ldots e_{\alpha_{5}}$ for some
$\lambda \in F - \{0\}$. Suppose there exists $M_{i} \neq LBT(f)$, such that
$\lambda_{i} \neq 0$. Suppose firstly that there exists a variable $x \in X$ such that $Deg_{x} M_{i} < Deg_{x}
LBT(f)$, then it is easy to note that none of the supports of any summand of $dom(\phi(M_{i}))$ contains $supp(\phi(x))$.

From now on we shall assume $Deg_{x} M_{i} = Deg_{x} LBT(f)$ for
every $x \in X$. In this case we have $beg M_{i} = beg LBT(f)$ does
not hold because $M_{i} \neq LBT(f)$. It turns out it is sufficient
to analyze the cases in which $beg M_{i} <_{lex-rig} beg LBT(f)$ or
$beg LBT(f) <_{lex-rig} beg M_{i}$.

\begin{description}
\item $beg M_{i} <_{lex-rig} beg LBT(f)$. In this case there exists a variable \newline $x \in Yyn(LBT(f)) \cup
Zyn(LBT(f))$ such that $Deg_{x} beg M_{i} < Deg_{x} beg LBT(f)$. Then we have $\phi(M_{i}) = 0$.

\item $beg LBT(f) <_{lex-rig} beg M_{i}$. In this case we have $M_{i}$
is not bad. Notice that $Deg_{x} beg(M_{i}) = Deg_{x} beg(LBT(f)) =
Deg_{x} beg(LT(f))$ for every $x \in Z - pr(z)(LT(f))$. If there
exists a variable $x \in Y$ such that $Deg_{x} beg(LBT(f)) > Deg_{x}
beg(M_{i})$, we have $\phi(M_{i}) = 0$. We claim that there always
exists $x \in Y$ such that $Deg_{x} beg(M_{i}) < Deg_{x}
beg(LBT(f))$. For, let us assume that $M_{i} \neq LT(f)$ (see
Observation \ref{observacao}). If $Deg_{x} beg(LBT(f)) \leq Deg_{x}
beg(M_{i})$ for every $x \in Y$, we have $Deg_{pr(z)(LT(f))}
beg(M_{i}) = Deg_{pr(z)(LT(f))} beg(LBT(f)) + 1$, because $M_{i}$ is
not bad. We recall that $beg M_{i} <_{lex-rig} beg LT(f)$, then
there exists $x_{1} \in Y$ such that $Deg_{x_{1}} beg(M_{i}) <
Deg_{x_{1}} beg(LT(f))$. On the other hand we have \[Deg_{x_{1}}
beg(M_{i}) \geq Deg_{x_{1}} beg(LBT(f)),\ Deg_{x_{1}} beg(LBT(f))
\geq Deg_{x_{1}} beg(LT(f))\] which is a contradiction. Hence
$\phi(M_{i}) = 0$.
\end{description}
In order to complete the proof of Case 3 it is sufficient to follow verbatim the proof of Case 1. We obtain $dom(\phi(f)) = \phi(LBT(f)) = 0$ which is a contradiction and we are done.
\end{proof}

Combining Lemma \ref{combss3} and Proposition \ref{proposicao 10.16} and in light of Lemma \ref{combssstar2} we have the following result.

\begin{teo}
For each $n\in\N$ and $n_1,n_2\in\N$ such that $n_1+n_2=n$ we have \[c_{(n_1,n_2)}(SS3)\leq c_{(n_1,n_2)}(E_{k},W(n))\leq c_{(n_1,n_2)}^\circ(SS3).\]
\end{teo}

\end{document}